\documentclass[10pt]{amsart}
\usepackage{a4wide}
\usepackage{amsfonts}
\usepackage{psfrag}

\usepackage{amsmath}
\usepackage{amssymb}
\usepackage{amsthm}
\usepackage{graphicx}

\newtheorem {thm}{Theorem}
\newtheorem {definition}[thm]{Definition}
\newtheorem{prop}[thm]{Proposition}
\newtheorem {lem}[thm]{Lemma}
\newtheorem {remark}[thm]{Remark}
\newtheorem {conj}[thm]{Conjecture}
\newtheorem {convention}[thm]{Convention}
\newtheorem {example}[thm]{Example}

\numberwithin{equation}{section}
\numberwithin{thm}{section}

\newcommand{\E}{\mathbb{E}}
\newcommand{\N}{\mathbb{N}}
\newcommand{\R}{\mathbb{R}}
\newcommand{\1}{\mathbf{1}}
\newcommand{\bu}{\mathbf{u}}
\newcommand{\bv}{\mathbf{v}}
\renewcommand{\P}{\mathbb{P}}

\newcommand{\Z}{\mathbb{Z}}

\begin{document}
\title[Finite and Infinite Rate Mutually Catalytic Branching]{Mutually Catalytic Branching Processes and Voter Processes with Strength of Opinion}
\author{Leif D\"oring}
\thanks{LD is supported by the Foundation Science Mat\'ematiques de Paris}
\address{LPMA, Universite Paris VI, Tours 16/26, 4 Place Jussieu, 75005 Paris}

\email{leif.doering@googlemail.com}
\author{Leonid Mytnik}
\thanks{LM is partly supported by the Israel Science Foundation and B. and G. Greenberg Research Fund (Ottawa) }
\address{Faculty of Industrial Engineering and Management Technion Israel Institute of Technology, Haifa 32000, Israel}
\email{leonid@ie.technion.ac.il}

\email{}
\thanks{}

\maketitle

	\begin{abstract}
		Since the seminal work of Dawson and Perkins, mutually catalytic versions of superprocesses have been studied frequently. In this article we combine two approaches extending their ideas: the approach of adding 		correlations to the driving noise of the system is combined with the approach of obtaining new processes by letting the branching rate tend to infinity. The processes are considered on a countable site space.
\\
		We introduce infinite rate symbiotic branching processes which surprisingly can be interpreted as generalized voter processes with additional strength of opinions. Since many of the arguments go along 		the lines 		of known proofs this article is written in the style of a review article.
	\end{abstract}
	\tableofcontents

	Going back to the seminal work of Watanabe \cite{W} and Dawson \cite{D78}, the subject of measure-valued diffusion processes arising as scaling limits of branching particle systems has attracted the interest of many probabilists. Many tools had to be developed to study the fascinating properties of the Dawson/Watanabe process (also called superprocess or super-Brownian motion) and its relatives. Characterizations and constructions of the process via a Laplace transform duality to non-linear parabolic partial differential equations, infinitesimal generator and corresponding martingale problem, the pathwise lookdown construction of Donelly/Kurtz \cite{DK1}, \cite{DK2} or Le Gall's Brownian snake construction based on the Ray-Knight theorems (see the overview \cite{LG}) led to many deep results. Much of the analysis is based on the branching property, i.e. the sum of two independent super-Brownian motions $X_t^1,X_t^2$ is equal in distribution to a single super-Brownian motion 
 started at $X^1_0+X^2_0$.
	
	\smallskip
	In the early 90s further directions became popular. Super-Brownian motion was found to be a universal scaling limit not only of branching systems but also of interacting particle systems such as voter process and its modifications  (see 	for instance \cite{CDP}, \cite{CP05}). Furthermore, instead of considering plainly super-Brownian motion, interactions were introduced. Tools such as Dawson's generalized Girsanov theorem \cite{D78} have been successfully applied in various contexts. Here, we will be mostly interested in variants of catalytic super-Brownian motion, i.e. super-Brownian motion with underlying branching mechanism depending on a catalytic random environment. As long as the environment is fixed, a good deal of the analysis can still be performed with techniques developed for the super-Brownian motion. More delicately, taking into account connections to stochastic heat equations, 
	Dawson/Perkins introduced a mutually catalytic superprocess (see \cite{DP98}). Their mutually catalytic branching model on the continuous site space $\R$ consists of two super-Brownian motions each being the catalyst of branching for the other.    The model was described  via stochastic 	heat equations.  They considered
	\begin{equation*}
		\begin{cases}
			\frac{\partial }{\partial t}u_t(x) = \frac{1}{2}\Delta u_t(x)+\sqrt{ \gamma u_t(x) v_t(x)} \, W^1_t(dt,dx),\quad t\geq 0, x\in \R,\\
			\frac{\partial }{\partial t}v_t(x) = \frac{1}{2}\Delta v_t(x)+\sqrt{ \gamma u_t(x) v_t(x)} \, W^2_t(dt,dx),\quad t\geq 0, x\in \R,
		\end{cases}
	\end{equation*}
	driven by two independent white noises $\dot{W}^1, \dot{W}^2$ on $\R^+\times \R$.
 Here, $\Delta$ denotes the one dimensional Laplacian. 
	 The mutually catalytic interaction of two super-Brownian motions has one particular drawback: the branching property is destroyed so that many of the previously known tools collapse.
	Fortunately, some ideas borrowed from the study of interacting particle systems and interacting diffusion models could be applied successfully due to the symmetric nature of the model. In particular, a self-duality that extends the 	linear system duality known for interacting particle systems could be established and utilized to prove uniqueness and longterm properties. Besides the above continuous model, the mutually catalytic model on the lattice was constructed and studied by Dawson and Perkins as well.
	
	\smallskip
	This article, which is focused on spatial branching models on discrete space, is motivated by two recent developments. First, in the series of papers \cite{KM10}, \cite{KM10a},  \cite{KM10b} the effect of sending the branching rate $\gamma$ to infinity was studied  in the discrete space mutually catalytic branching model. 	The resulting  infinite rate 	mutually catalytic branching model is one of the rare tractable spatial models with finite $2-\epsilon$ moments but infinite $2$nd moment forcing the system to have 		critical scaling behavior.\\
	Secondly, Etheridge/Fleischmann \cite{EF04} introduced  the following generalization of 
   the Dawson-Perkins model. They considered the 
mutually catalytic branching model with correlated driving noises which, on the level of a branching system approximation, corresponds to a two type system of 	branching particles with correlated branching mechanism. They called their model symbiotic branching model in contrast to the mutually catalytic branching model of Dawson/Perkins that appears as a special case for zero 		correlations. We will use equally the name symbiotic branching and mutually catalytic branching with correlations.	
	Correlating the branching mechanism might seem artificial on first 	view. On second view one observes that the extremal correlations lead to well-known models from the theory of interacting diffusion models: the stepping stone model with 		applications in theoretical biology and a parabolic Anderson model with applications in statistical physics. As those models  have very different path behavior one could expect phase-transitions occurring when changing the 			correlations. On the level of moments those phase transitions have been revealed recently in \cite{BDE11}: there is a precise transition for $2$nd moments when the correlation parameter changes from negative to positive.
	
	\smallskip
	The main result of this article, formulated here in a slightly simplified version, is the following theorem which should be viewed as the natural combination of the two aforementioned developments. In particular, the theorem below  extends results from \cite{KM10b} to the case 
of ``correlated (symbiotic) branching''. 
	\begin{thm}\label{0}
		Suppose $\varrho\in (-1,1)$ is a parameter and $(u_t^\gamma,v_t^\gamma)_{t\geq 0}$ is the unique non-negative weak solution to the symbiotic branching model on the lattice defined by
			\begin{align}\label{ss}
			\begin{cases}
				d u_t(k)=\Delta u_t(k)\,dt+\sqrt{\gamma v_t(k)u_t(k)}\,dB^1_t(k),\quad k\in\Z^d,\\
				d v_t(k)=\Delta v_t(k)\,dt+\sqrt{\gamma v_t(k)u_t(k)}\,dB^2_t(k),\quad k\in \Z^d.
			\end{cases}
		\end{align}
		 Here, $\Delta$ denotes the discrete Laplacian on $\Z^d$
		 \begin{align*}
		 	\Delta f(k)=\sum_{|i-k|=1}\frac{1}{2d} (f(i)-f(k))
		 \end{align*}	
		 and the driving Gaussian process $\{B^1_\cdot(k),B^2_\cdot(k)\}_{k\in\Z^d}$ has correlation structure
		\begin{align}\label{ss2}\begin{split}
			\E\big[B^1_{t}(k)B^1_{t}(j)\big]&=\delta_0(k-j)t,\\
			\E\big[B^2_{t}(k)B^2_{t}(j)\big]&=\delta_0(k-j)t,\\
			\E\big[B^1_{t}(k)B^2_{t}(j)\big]&=\varrho\delta_0(k-j)t.\end{split}
		\end{align}
		Additionally, assume that the non-negative initial conditions $(u_0^\gamma,v_0^\gamma)=(U_0,V_0)$ do not depend on $\gamma$, satisfy a minor growth condition (for the precise definitions see~(\ref{111}) and Section~\ref{sec:2.1.1}) and also
		\begin{align*}
			(U_0(k),V_0(k))\in E:=\big\{(y_1,0):y_1\geq 0\big\}\cup \big\{ (0,y_2):y_2\geq 0\big\}\subset \R^2
		\end{align*}
		for all $k\in \Z^d$.
		\smallskip
		
		Then $(u^\gamma,v^\gamma)$ converges, as $\gamma$ tends to infinity, weakly in the Meyer-Zheng ``pseudo-path" topology (introduced  in~\cite{MZ}),  to a limiting RCLL process $(U_t,V_t)_{t\geq 0}$ taking values in $E$
		which is the unique weak solution to the system of Poissonian integral equations
		\begin{align}\label{UV}\begin{split}
			U_t(k)&=U_0(k)+\int_0^t\Delta U_s(k)\,ds\\
				&+\int_0^t\int_0^{I_s(k)}\int_{E}\big(y_2 V_{s-}(k)+(y_1-1)U_{s-}(k)\big)(\mathcal{N-N'})(\{k\},d(y_1,y_2),dr,ds),\\
			V_t(k)&=V_0(k)+\int_0^t\Delta V_s(k)\,ds\\
				&+\int_0^t\int_0^{I_s(k)}\int_{E}\big(y_2 U_{s-}(k)+(y_1-1)V_{s-}(k)\big)(\mathcal{N-N'})(\{k\},d(y_1,y_2),dr,ds),\end{split}
		\end{align}
		for $t\geq 0$ and $k\in \Z^d$.	Here,
${\mathcal N}$ is a  Poisson point measure on $\Z^d\times E\times (0,\infty)\times (0,\infty)$ with intensity measure
		\begin{align*}
			\mathcal N'(\{k\},d(y_1,y_2),dr,ds)=\nu^\varrho(d(y_1,y_2))\,dr\,ds,\;\;\forall k\in \Z^d,
		\end{align*}
		where
		\begin{align*}
			\nu^\varrho(d(y_1,y_2))=\begin{cases}
				{p(\varrho)^2}\sqrt{1-\varrho^2}\frac{y_1^{p(\varrho)-1}}{\pi\big(y_1^{p(\varrho)}-1\big)^2}\,du&: y_2=0,\\
				{p(\varrho)^2}\sqrt{1-\varrho^2} \frac{y_2^{p(\varrho)-1}}{\pi\big(y_2^{p(\varrho)}+1\big)^2}\,dv&: y_1=0,
				\end{cases}
		\end{align*}
		\begin{align*}
		p(\varrho)=\frac{\pi}{\frac{\pi}{2}+\arctan\big(\frac{\varrho}{\sqrt{1-\varrho^2}}\big)}\,,
		\end{align*}
		 and, for any $k\in \Z^d$,  
	\begin{align*}
		I_t(k)&=\begin{cases}
		        	\frac{\Delta V_{t-}(k)}{U_{t-}(k)}&:U_{t-}(k)>0,\\
		        	\frac{\Delta U_{t-}(k)}{V_{t-}(k)}&:V_{t-}(k)>0,\\
			0&:U_{t-}(k)=V_{t-}(k)=0.
		       \end{cases}
	\end{align*}
\end{thm}
\begin{remark}
Theorem \ref{0} will be proved in Section \ref{sec:3} for more general countable 
  state-space $S$ instead of $\Z^d$ and $Q$-matrix $\mathcal A$ instead of $\Delta$. The proof of the theorem follows from Theorems \ref{pro:1} and \ref{thm:111}.
\end{remark}
The parameter $\varrho$ only occurs in the measure $\nu^\varrho$ so that it does not surprise that  proofs go along the lines of \cite{KM10b} replacing in their Poissonian equations $p(0)=2$ by some  $p(\varrho)>1$. The striking fact of the generalization to $\varrho\neq 0$ is that it allows to understand $(U,V)$ as a family of generalized voter processes with the standard voter process appearing for $\varrho=-1$.

The generalized voter process interpretation goes as follows: Suppose at each site $k\in \Z^d$ lives a voter with one of two possible \textbf{opinions}. Their opinions additionally have a non-negative \textbf{strength}. 
		Mathematically speaking, the type of opinion is determined by the non-zero coordinate of the 
		opinion-vector (recall the definition of $E$) and the strength is determined by the absolute value, i.e.
		\begin{itemize}
			\item $(u,0)\in E$ codes opinion $1$ of strength $u$,
			\item $(0,v)\in E$ codes opinion $2$ of strength $v$.
		\end{itemize}
		Formulated like this, the standard voter process only takes values $(1,0)$ and $(0,1)$ since all opinions do have a fixed strength, say $1$.
		If $u$ (resp. $v$) is large, we say the opinion is strong, otherwise weak.\\ Voters change dynamically their opinions and their strength according to the next two possibilities:
		\begin{itemize}
			\item \textbf{Change of opinion strength only}: Suppose $\mathcal N$ has an atom at $(k,(y_1,0),s,r)$. Then, by definition of the two integrands, the Poissonian integrals produce two-dimensional jumps of the form
				\begin{align*}
					0 \left(V_{s-}(k)\atop U_{s-}(k)\right)+(y_1-1)\left(U_{s-}(k)\atop V_{s-}(k)\right)
				\end{align*}
				so that, added to the current state of the system, the state of the system at site $k$ changes according to
				\begin{align*}
					\left(U_{s-}(k)\atop V_{s-}(k)\right)\mapsto y_1\left(U_{s-}(k)\atop V_{s-}(k)\right).
				\end{align*}
				If before the jump the voter had opinion $1$ of strength $u$, the change is $(u,0)\mapsto (y_1 u,0)$ and $(0,v)\mapsto (0,y_1v)$ if the voter had opinion $2$ before. Hence, if $(y_1,0)$ is chosen by 				the basic jump measure $\nu^\varrho$, only the strength of the opinion changes but not the type.
			\item \textbf{Change of opinion and its strength}: Suppose $\mathcal N$ has an atom at $(k,(0,y_2),s,r)$.  Then, by definition of the integrands, the Poissonian integrals produce jumps of the form
				\begin{align*}
					y_2\left(V_{s-}(k)\atop U_{s-}(k)\right)+(0-1)\left(U_{s-}(k)\atop V_{s-}(k)\right)
				\end{align*}
				so that, added to the current state of the system, the state of the system at site $k$ changes according to
				\begin{align*}
					\left(U_{s-}(k)\atop V_{s-}(k)\right)\mapsto y_2\left(V_{s-}(k)\atop U_{s-}(k)\right).
				\end{align*}
				If before the jump the voter had opinion $1$ of strength $u$, the change is $(u,0)\mapsto (0,y_2 u)$ and $(0,v)\mapsto (y_2v,0)$ if the voter had opinion $2$ before. Hence, if $(0,y_2)$ is chosen by 				the basic jump measure $\nu^\varrho$, the voter changes strength and type of opinion.
		\end{itemize}
		\begin{remark}
		 We show in Section \ref{sec:voter2} that Theorem \ref{0} extends naturally to $\varrho=-1$ when $\nu^\varrho$ is replaced by $\delta_{(0,1)}$. If additionally $(U_0,V_0)\in \{(0,1),(1,0)\}^{\Z^d}$, then solutions to (\ref{UV}) give standard voter processes.		Note that in this case only the second type of changes occurs since $\mathcal N$ only has atoms at $(k,(0,1),s,r)$. Hence, the strength of the opinion does not change. In particular, we only see opinion changes from $(1,0)$ to $(0,1)$ and vice 		versa.
		\end{remark}	

		Finally, we should also give an interpretation to the rates $I_t(k)$: due to the definition of $I_t(k)$ and $\Delta$, the rate of change for the voter at site $k$ is high if the strength of the opinions of his neighbors of different opinion is high compared to his opinion.
		 In particular, voters with weak conviction tend to change quicker their opinions than voters with strong conviction.
\medskip

	The result of Theorem \ref{0} might look frightening to the reader not familiar with interacting diffusion processes and/or jump diffusions. However, once the connection to the results of \cite{KM10b} and \cite{BDE11} is understood, the proofs of 	the theorem go along the lines of \cite{KM10b}. Therefore, we decided to write this article in the form of a review article explaining in depth the background. We do not give many detailed proofs but instead give more detailed calculations to explain the origins of (\ref{UV}). In the following we explain carefully
	\begin{itemize}
		\item the background of catalytic branching processes,
		\item definitions, existence, uniqueness and tools for (\ref{ss}),
		\item what is known on the longtime behavior of (\ref{ss}) to motivate the choice of $\nu^\varrho$ in the theorem via planar Brownian motions exiting a cone,
		\item more details on (\ref{UV}) and (alternative) constructions of $(U,V)$,
		\item concepts and definitions for jump diffusions.
	\end{itemize}
		 The background and connections to well-known stochastic processes from the literature will be explained exhaustively in \textbf{Section \ref{sec:1}}. Two different routes from known models to mutually catalytic branching models are disclosed: the original motivation of Dawson/Perkins originating from catalytic super-Brownian motion and symbiotic branching as unifying model for some interacting diffusions. As a final motivation the connection of stepping stone processes and voter processes is recalled. \textbf{Section~\ref{sec:2}} is devoted to an overview of precise definitions, existence and 	uniqueness results and longtime properties for finite rate symbiotic branching 		processes. In particular, the second moment transitions are discussed in detail. Proofs are cooked down to the main ingredients. Finally, in \textbf{Section \ref{sec:3}} the infinite rate symbiotic branching processes are introduced and reinterpreted as generalized voter processes in the very end. Additionally, a brief summary of jump diffusions is included to the appendix.
	
\section{Background and Motivation}\label{sec:1}
\subsection{From Superprocesses to Mutually Catalytic Branching}
	Being a major subject of probability theory, measure-valued diffusions, or superprocesses, such as super-Brownian motion and the Fleming-Viot process have been well studied during the last three decades. Important properties of superprocesses have been proved and connections to other areas of mathematics such as partial differential equations have been established. 
 For a detailed exposition of the subject  the reader is referred to~\cite{dawson}, \cite{Perkins} and \cite{E00}.
	\smallskip
	
	Here we introduce briefly super-random walks - the spatially discrete analogues of super-Brownian motion.
   Studying these processes gave a strong motivation  to investigate spatial branching processes with interactions, and in particular, 
  mutually catalytic branching processes on discrete space - the main theme of this article. 
	To introduce super-random walks, we  start with the following approximating particle system.
   Assume that an initial configuration 
 of a large number  (of order $N$) of particles distributed over $\Z^d$ is given. 
The particles move as independent simple random walk in $\Z^d$ and each particle independently of the others dies after an exponential time of rate $\gamma N$, with $\gamma>0$, and at the place of death it leaves a random number of offspring particles, drawn from a fixed integer valued law $\mu$. The particles of the updated population continue their motion and reproduction according to the same rules. 
This process is usually referred to as a branching random walk with the branching law $\mu$ and we will assume in the sequel that $\mu$ has expectation $1$ (this means criticality) and finite variance $\sigma^2>0$. 
The process $X_t^{(N)}$ is then defined to be the finite atomic measure which loosely speaking gives measure of mass $1$  
to each particle alive at time $t$. To be more precise,
$$ X_t^{(N)}= \sum_{i}\delta_{x_i},\;\;t\geq 0,$$
where $x_i$ is a position of the $i$-th particle alive at time $t$. Assume that, as $N$ tends to infinity, $\frac{1}{N}X_0^{(N)}$ converges weakly in the space of finite measures on $\Z^d$ to a  measure $X_0$. Then one can show that the measure-valued process 
	$\{\frac{1}{N}X_{t}^{(N)}\}_{t\geq 0}$  converges weakly to a limiting measure-valued process $\{X_t\}_{t\geq 0}$ which is
 called super-random walk and is uniquely characterized via the following martingale problem: for bounded test-functions $\phi:\Z^d\to \R$
	\begin{align*}
		M_t(\phi):=\int_0^t\int_{\Z^d} \frac{1}{2}\Delta\phi(x)X_s(dx)\,ds- \int_{\Z^d} \phi(x) X_0(dx)
	\end{align*}
	is a square-integrable martingale with quadratic variation process
$\sigma^2\gamma \int_0^t \int_{\Z^d}\phi^2(x)  X_s(dx)ds.$
	Here, $\Delta$ denotes the discrete Laplace operator as defined in Theorem \ref{0}. 
   An interesting observation is the following invariance property: irrespectively of $\mu$, the finite variance assumption for the branching mechanism $\mu$ leads to a universal limit depending only on the variance $\sigma^2$ and the parameter $\gamma$ which is also called the branching rate. In what follows, we assume $\sigma^2=1$. 	
   It is worth mentioning that if we ignore the spatial motion and count just the total number of particles, the scaling procedure is nothing else but the scaling of critical and finite variance Galton-Watson processes which 
  leads  towards  classical  Feller's branching diffusion
	\begin{align*}
		dZ_t=\sqrt{\gamma Z_t}\,dB_t\,,
	\end{align*}
   where $Z_t=X_t(\Z^d), \; t\geq 0$.

Note that super-random walks can be characterized  as solutions to stochastic differential equations.
Abbreviating $u_t(k)=X_t(\{k\})$, the super-random walk is a weak solution to 
following system of stochastic differential equations (which is, in fact, a discrete version of a
stochastic heat equation)
	\begin{align}\label{e}
		d u_t(k)=\Delta u_t(k)\,dt+\sqrt{\gamma u_t(k)}\,d B_t(k),\quad k\in \Z^d,
	\end{align}
	where $\{B(k)\}_{k\in \Z^d}$ is a collection of  independent Brownian motions.
 Next, we  proceed to a more recent development: measure-valued processes with interactions.  One way to introduce interaction into the model 
is to replace the constant branching rate $\gamma$ in the particle approximation by a random, adapted and space-time varying branching rate $\gamma(t,k,\omega)$, also called the catalyst. Some particular choices of 			branching environments $\gamma$ and related models over continuous space have been discussed in the literature (see for instance \cite{DF94}, \cite{DF95}, \cite{D96}). For example, one can consider a super-random walk on 
	$\Z^d$ in a super-random walk environment. Building upon (\ref{e}), this model can be described as a solution to the following system of stochastic differential equations:	
	\begin{align}\label{f}
		\begin{cases}
			d u_t(k)=\Delta u_t(k)\,dt+\sqrt{v_t(k)u_t(k)}\,d B^1_t(k),\quad k\in \Z^d,\\
			d v_t(k)=\Delta v_t(k)\,dt+\sqrt{\gamma v_t(k)}\,d B^2_t(k),\quad k\in\Z^d,
		\end{cases}
	\end{align}
	driven by independent families of independent Brownian motions. A solution $(u_t)_{t\geq 0}$ is called super-random walk in the catalytic super-random walk environment $\gamma(\cdot,k,\omega)=v_\cdot(k)(\omega)$. 
	Note that (\ref{f}) describes the so-called one-way interaction model: the $v$-population catalyzes the $u$-population. Then the  natural extension of (\ref{f}) to two-way interaction is  the following mutually catalytic model. 
	\begin{definition}
		In the following, weak solutions $(u_t,v_t)_{t\geq 0}$, on a stochastic basis $(\Omega, \mathcal F,(\mathcal F_t)_{t\geq 0},\P)$, to the infinite system of stochastic differential equations (\ref{ss}) driven by independent Brownian motions
		will be called mutually catalytic branching processes with initial conditions $u_0,v_0$ and branching rate $\gamma>0$. To abbreviate, solutions will be denoted by \textbf{$\textrm{MCB}_\gamma$}.
In the sequel $\textrm{MCB}_\gamma$ will also denote a mutually catalytic branching process defined on a more general state space $S$ instead of $\Z^d$ and with $Q$-matrix $\mathcal A $ instead of $\Delta$.
\end{definition}

	It is easy to see that the branching property fails for $\textrm{MCB}_\gamma$. Hence, many of the classical tools developed for superprocesses also fail. Nonetheless, the simple symmetric choice of the interaction between 	$u$ and $v$ makes this mutually catalytic system tractable. 
	\begin{convention}
		In order to stress the underlying branching processes, the two components will be called types.
	\end{convention}
	As an example for the convention, if $u_t(k)=0$ for all $k\in \Z^d$ we will say that the first type died out.
		
\subsection{From Interacting Diffusions to Symbiotic Branching}\label{sec:inter}
	Interestingly, the study of mutually catalytic branching processes can also be motivated by the study of interacting diffusion processes. Given a family of independent Brownian motions $\{B_t(k)\}_{k\in\Z^d}$ and some function $f$ to be specified below, discrete-space parabolic stochastic partial differential equations
	\begin{eqnarray}\label{int}
	\begin{cases}
		d w_t(k)=\Delta w_t(k)\,dt + \sqrt{\gamma f(w_t(k))}\,dB_t(k),\\
		w_0(k)\geq 0,\quad k\in \Z^d,
	\end{cases}
	\end{eqnarray}
	have been studied extensively in the literature. Some prominent examples will be briefly discussed in the sequel. 
	\begin{example} \label{ex3}
		For $f(x)=x$ solutions of (\ref{int}) are super-random walks.
	\end{example}	
	This example has already been dealt with in detail in the previous subsection. 
	\begin{example}\label{ex1}
		For $f(x)=x(1-x)$, Equation (\ref{int}) is called stepping stone model.
	\end{example}
	In fact, the stepping stone model is the spatial generalization of the one-dimensional Wright-Fisher diffusion
	\begin{align}\label{wf}
		dX_t=\sqrt{\gamma X_t(1-X_t)}\,dB_t
	\end{align}
	that arises as a scaling limit of the Moran model in population genetics similarly as the Feller diffusion arises as a scaling limit of critical Galton-Watson processes. In contrast to the Galton-Watson model, the Moran model 	is not used to model the total number of individuals but instead counts the proportion of one allele in a diploid population for a fixed number of individuals. In particular, this interpretation corresponds to the solution of  (\ref{wf}) taking values in 	$[0,1]$ with absorption at $1$ or $0$ interpreted as fixation of genetic types. For an introduction to the questions of mathematical population genetics we refer to the lecture notes \cite{E10}. The stepping 	stone model of Example \ref{ex1} can be seen as an island version of the Wright-Fisher diffusion, i.e. additionally to the change of alleles, individuals live on islands which they change according to a nearest 		neighbor random walk.
	
		\smallskip
	Changing the scope once more, we have a look at statistical physics. Given a random field $\eta_t(k)$, possibly time-inhomogeneous, the discrete heat equation with random potential $\eta$
	\begin{align}\label{pam2}
	\begin{cases}
		\frac{\partial}{\partial t} u_t(k)=\Delta u_t(k)+\eta_t(k) u_t(k),\\
		u_0(k)\geq 0,\quad k\in \Z^d,
	\end{cases}
	\end{align}
	has attracted a lot of interest. It is usually referred to as a parabolic Anderson model. Again, there is a connection to a branching particle system: started at localized initial condition $u_0={\textbf 1}_{\{0\}}$, $u_t(k)$ is the 	expected number of particles in the system where one particle starts at $0$ and branches binary according to the breeding potential $\eta$. In particular in the case of time-independent iid random potential a detailed analysis 	of the behavior of solutions is possible; we refer to the overview article \cite{GK05}.
	If $\eta$ is the white noise case, then (\ref{pam2}) is a particular case of (\ref{int}) leading us to the next example.
	\begin{example}\label{pam}
		For $f(x)=x^2$, Equation (\ref{int}) describes the parabolic Anderson model with Brownian potential (white noise potential).
	\end{example}
	A detailed analytic study of the longtime behavior for this model can be found in the monograph \cite{CM}. For the probabilistic approach based on an explicit Feynman-Kac representation we refer to \cite{GdH07} and references therein. 

		\smallskip
	Finally, the simplest example should be mentioned. Already in this case, a non-trivial interplay of noise and drift can be observed (see \cite{CK00}).
	\begin{example}
		Choosing $f(x)=1/\gamma$, Equation (\ref{int}) describes interacting Brownian motions.
	\end{example}

	Now, as we have discussed examples that are of very different nature in terms of their origins and also of their properties, we should explain the connections to mutually catalytic models. Here is a preliminary definition for the two types interacting diffusion model introduced by Etheridge/Fleischmann in \cite{EF04}. A more precise and more general definition is given in Section 2.

	\begin{definition}\label{def1}
		In the following, weak solutions $(u_t,v_t)_{t\geq 0}$ on a stochastic basis $(\Omega, \mathcal F, (\mathcal F_t)_{t\geq 0},\P)$ to the infinite system of stochastic differential equations defined in (\ref{ss}) 
		driven by Brownian motions with correlation structure (\ref{ss2}) are called symbiotic branching processes with initial conditions $u_0,v_0$, branching rate $\gamma>0$ and correlation $\varrho\in [-1,1]$.\\ To abbreviate, the system of equations (\ref{ss}) and their solutions will be denoted by \textbf{$\textrm{SBM}_\gamma(\varrho)$} or just \textbf{$\textrm{SBM}_\gamma$}.
	\end{definition}

	The name symbiotic branching model was used in \cite{EF04} in order to stress the biological interpretation of the mutually catalytic behavior; the solution processes $u_t$ and $v_t$ might be considered as the distribution in space of two types.
	\begin{convention}
		For later use let us capture the correlation structure used for symbiotic branching in a name. We will say that two Brownian motions  satisfying $\E\big[B^1_tB^2_t\big]=\varrho t$, $t\geq 0$, are $\varrho$-correlated.
	\end{convention}
	
	\smallskip
	Having introduced the basic equations of this article, their relevance is emphasized by the following observation due to \cite{EF04}. For correlation $\varrho=0$, solutions of the symbiotic branching model are solutions of the mutually catalytic branching model $\textrm{MCB}_\gamma$. \\
	The case $\varrho=-1$ with the additional assumption $u_0+v_0\equiv 1$ corresponds to the stepping stone model. To see this, observe that in the perfectly negatively 	correlated case $B^1(i)=-B^2(i)$ which implies that the sum $u+v$ solves a discrete heat equation and with the further assumption $u_0+v_0\equiv 1$ stays constant for all time. Hence, for all $t\geq 0$, 
	$u_t\equiv 1-v_t$ which shows that $u$ is a solution of the stepping stone model with initial condition $u_0$ and $v$ is a solution with initial condition $v_0$. \\
	Finally, suppose $w$ is a solution of the parabolic 	Anderson model, then, for $\varrho=1$, the pair $(u,v):=(w,w)$ is a solution of the symbiotic branching model with initial conditions $u_0=v_0=w_0$ as now $B^1(i)=B^2(i)$.
	

\subsection{Infinite Rate Symbiotic Branching Processes and Voter Processes I}\label{sec:voter1}
	To motivate the procedure of sending $\gamma$ to infinity in Theorem \ref{0} and to highlight for a first time why the generalized voter processes appear as limits, 
let us briefly discuss the voter process and its connection to the stepping stone model. For extensive information about interacting particle systems we refer to the monograph of Liggett \cite{L}.
	
	\smallskip
	A way of defining interacting particle systems is a description via infinitesimal generators. Here, we assume that the voters live on $\Z^d$ and communicate only with their nearest neighbors. To define the dynamics via a 		generator, the 	state-space 
	\begin{align*}
		\Sigma=\{0,1\}^{\Z^d}
	\end{align*}
	is fixed. The generator acts via
	\begin{align}\label{gener}
		A f(\eta)=\sum_{k\in \Z^d}c(k,\eta)\big(f(\eta^{(k)})-f(\eta)\big),
	\end{align}
	on test-functions $f:\Sigma\to \R$ only depending on finitely many coordinates and $\eta^{(k)}$ is defined to be the configuration in which the opinion is flipped only at site $k$ and the rate of change at site $k$ is proportional  (normalized to total rate $1$) to 		the number of neighbors with different opinion:
	\begin{align*}
		c(k,\eta)=\frac{1}{2d}\sum_{|i-k|=1}\mathbf 1_{\{\eta(i)\neq \eta(k)\}}.
	\end{align*}
	Interestingly, the analysis of the longtime behavior of a voter process is drastically simplified by a pathwise graphical construction (see \cite{Dur}, page 129): for each voter a vertical line is drawn downwards and each line carries a Poisson 	process firing tacks on that line. At each tack, a horizontal line is drawn randomly to a neighbor. With an initial configuration $U_0\in \Sigma$, the construction goes as follows: for each site $k$ with $U_0(k)=1$ water is filled 	into the vertical line and disperses downwards. Whenever there is an arrow pointing away from the line (this corresponds to persuading a neighbor) the water goes on downwards and, additionally, flows through the arrow to 		disperse downwards in the neighbor's line. When an arrow points from a neighbor's line towards the voter's line the water stops (this corresponds to be persuaded by a neighbor). At time $t\geq 0$, the configuration $U_t$ is defined as follows: all sites filled by water carry a $1$ and al
 l others a $0$. It is heuristically clear that this construction yields a Markov process with generator (\ref{gener}) but interestingly it simultaneously gives a useful dual relation: reversing time and using the 	same arrows in the opposite direction, the resulting process is a system of instantaneously coalescing random walks. A simple consequence of this construction is a moment formula for the voter process:
	\begin{align}\label{mmm}
		\E^{U_0}\big[U_t(k_1)\cdots U_t(k_m)\big]&=\E\Big[\prod_i U_0(\xi_t^i)\,\Big|\,\xi^1_0=k_1,\cdots,\xi^m_0=k_m\Big],
	\end{align}
	where $\xi^1,...,\xi^m$ are independent simple random walks that coalesce instantaneously when colliding.
	The product runs over all non-coalesced random walks at time $t$.\\
	
	Now, let us return to the stepping stone model
	\begin{align}\label{bbbb}
			d w_t(k)=\Delta w_t(k)\,dt + \sqrt{\gamma w_t(k)(1-w_t(k))}\,dB_t(k)
	\end{align}
	that was already identified to the symbiotic branching process s $\varrho=-1$. Unfortunately, there is no direct graphical construction for the stepping stone model, but still, a moment representation similar to (\ref{mmm}) was derived in \cite{S80}: suppose the $\xi^i$ are as above but 	now two particles coalesce when they have spent together an independent exponential time of rate $\gamma$. More precisely, suppose $w^\gamma$ is a solution of (\ref{bbbb}) with initial conditions $w^\gamma_0$. Then
	\begin{align*}
		\E\big[w^\gamma_t(k_1)\cdots w^\gamma_t(k_m)\big]=\E\Big[\prod_i w^\gamma_0(\xi_t^i)\,\Big|\,\xi^1_0=k_1,\cdots,\xi^n_0=k_m\Big],
	\end{align*}
	where again the product runs over all random walks alive at time $t$. Sending $\gamma$ to infinity for the stepping stone model and assuming that $w_0^\gamma=U_0\in \Sigma$ does not depend on $\gamma$, we now observe that
	\begin{align*}
		\lim_{\gamma\to\infty}\E^{w^\gamma_0}\big[w^\gamma_t(k_1)\cdots w^\gamma_t(k_m)\big]=\E^{U_0}\big[U_t(k_1)\cdots U_t(k_m)\big],
	\end{align*}	
	since only the coalescence mechanism has changed: random walks now coalesce instantaneously after colliding.
	Boundedness of solutions implies that convergence of the moments suffices to deduce convergence of the finite dimensional distributions so that the \textbf{infinite rate limit of the stepping stone model is nothing but the standard voter 	process}. For more on this we refer to Section 10.3.1 of \cite{dawson}. 
	
\newpage
\section{Finite Rate Symbiotic Branching Processes}\label{sec:2}	
	
	The aim of this section is to give a compressed overview of definitions and results for symbiotic branching processes $\textrm{SBM}_\gamma$ with finite branching rate. After introducing some notation, precise definitions and a sketch of existence and 			uniqueness proofs we turn our focus to the longtime behavior. Let
	\begin{align}\label{exitd}
		Q_{u,v}^\varrho:=\big(W^1_\tau,W^2_\tau\big)
	\end{align}
	be the exit law	of a pair of $\varrho$-correlated Brownian motions started in $(u,v)$ for some $u\geq 0, v\geq 0,$ stopped at the exit-time 
	\begin{align}\label{exitt}
		\tau=\inf\big\{t: W^1_tW^2_t=0\big\}.
	\end{align}
	The laws $Q_{u,v}^\varrho$ are concentrated on the boundary of the first quadrant which we denote by
	\begin{align*}
		E=\big\{(y_1,0):y_1\geq 0\big\}\cup \big\{ (0,y_2):y_2\geq 0\big\}\subset \R^2.
	\end{align*}
	Whenever the initial condition $(u,v)$ is not crucial we abbreviate the exit-law as $Q^\varrho$. We present in the sequel those results on the longtime behavior of symbiotic branching which are related to $Q^\varrho$. Those will serve as preparation for the study of infinite rate symbiotic branching processes which we will denote by $\textrm{SBM}_\infty$.

\subsection{Existence, Uniqueness and Tools}\label{sec:uni}

	Recall Definition \ref{def1}, where we defined $\textrm{SBM}_\gamma$ as a system of coupled stochastic differential equations with drift operator $\Delta$. With some technical complications, $\Z^d$ can be replaced by a countable 	set $S$ and $\Delta$ by an operator
	\begin{align*}
		\mathcal A w(i)=\sum_{j\in S}a(i,j) w(j),
	\end{align*}
	where $\big(a(i,j)\big)_{i,j\in S}$ is the $Q$-matrix of a symmetric $S$-valued Markov process with uniformly bounded jump-rates. The particular case of $\Delta$ occurs for the choice $S=\Z^d$ and $a(i,j)=\frac{1}{2d}$ if $|i-j|=1$.

\subsubsection{State Spaces}
\label{sec:2.1.1}
	Let us define an infinite dimensional state-space for solutions which is commonly used in studying interacting particle systems. To do so, suppose $\beta:S\to \R^+$ is such that
	\begin{align*}
		\sum_{i\in S}\beta(i)<\infty\qquad\text{and}\qquad \sum_{i\in S}\beta(i)|a(i,k)|<M\beta(k)
	\end{align*}
	for all $k\in S$.  The state-space for the two-type model $\textrm{SBM}_\gamma$ then consists of pairs of sequences that grow slowly enough compared to $\beta$:
	\begin{align*}
		L^2_\beta=\big\{(u,v): S\to \R^+\times \R^+\,\text{ s.t. }\, \langle u,\beta\rangle<\infty\text{ and }\langle v,\beta\rangle<\infty\big\},
	\end{align*}
	where $\langle f,g\rangle=\sum_{k\in S}f(k)g(k)$. $L^2_\beta$ is equipped with the topology induced by the norm $||(u,v)||_\beta=\langle |u|+|v|,\beta\rangle$.
	 Existence of such a sequence $\beta$ is ensured by Lemma IX.1.6 of \cite{L}. 
		In the following we fix a test-sequence $\beta$ and only work on the
corresponding fixed state-space $L^2_\beta$.

\subsubsection{Precise Definition and Existence of Solutions}
	Having defined proper state-spaces, we can give the precise definition of solutions to $\textrm{SBM}_\gamma$.

		\begin{definition}\label{defsol} For $(u_0,v_0)\in L_\beta^2$, we say that $(u_t,v_t)_{t\geq 0}$, more precisely $(u,v,B^1,B^2)$, is a (weak) solution of $\mathrm{SBM}_\gamma$ on the filtered probability space $(\Omega, \mathcal{F}, (\mathcal{F}_t)_{t\geq 0},\P)$ if
		\begin{itemize}
			\item[i)] $\big\{B^1_\cdot(i),B^2_\cdot(i)\big\}_{i \in S}$ is a set of $(\mathcal{F}_t)$-adapted Brownian motions satisfying for $t>0$
			\begin{align*}
				\E\big[B^1_{t}(k)B^1_{t}(j)\big]&=\delta_0(k-j)t,\\
				\E\big[B^2_{t}(k)B^2_{t}(j)\big]&=\delta_0(k-j)t,\\
				\E\big[B^1_{t}(k)B^2_{t}(j)\big]&=\varrho\delta_0(k-j)t,
			\end{align*}
			\item[ii)] $u_\cdot,v_\cdot$ are $(\mathcal{F}_t)$-adapted stochastic processes, almost surely satisfying the integral equations
				\begin{align*}
					u_t(k)&=u_0(k)+\int_0^t\mathcal A u_s(k)\,ds+\int_0^t\sqrt{\gamma u_s(k)v_s(k)}dB_s^1(k),\\
					v_t(k)&=v_0(k)+\int_0^t\mathcal A v_s(k)\,ds+\int_0^t\sqrt{\gamma u_s(k)v_s(k)}dB_s^2(k),
				\end{align*}
				for $k\in S$,
			\item[iii)] $(u_\cdot,v_\cdot)$ is almost surely continuous with $(u_t,v_t)\in L_\beta^2$ for all $t\geq 0$.
		\end{itemize}
	\end{definition}

		
We now give a quick glance on how to construct solutions for $\textrm{SBM}_\gamma$. For a very detailed proof for $\varrho=0$ we refer the reader to \cite{DP98}. \textbf{From now on, until the end of the section, we assume $S=\Z^d$ and $\mathcal A=\Delta$.} The relations to the exit-law of $\varrho$-correlated Brownian 	motions remain unchanged in this simplified setting so that it serves equally well as a preparation for $\textrm{SBM}_\infty$.

	\begin{thm}\label{ex:dis}
		If $(u_0,v_0)\in L_\beta^2$, there is a weak solution of $\textrm{SBM}_\gamma$.
	\end{thm}
	
	\begin{proof}[Sketch of Proof]
	The proof goes along famous arguments due to \cite{SS80} based on finite dimensional SDE theory and limit considerations. Cutting the infinite index set, solutions  to the finite system can be constructed and then, by moment estimates, their convergence to a weak solution of $\textrm{SBM}_\gamma$ can be shown.\\
		For positive integers $n$, let $S_n=\Z^d\cap [-n,n]^d$ be a finite subset of $\Z^d$. To define the approximating system, we consider the following system of finite-dimensional stochastic differential equations which we denote 		by $\mathrm{SBM}^n_\gamma$:
		\begin{align*}
			u_t^n(k)&=u_0(k)+\int_0^t\sum_{\underset{|j-k|=1}{ j\in S_n}}\frac{1}{2d}(u^n(j)-u^n(k)) \,ds+\int_0^t\sqrt{\gamma u_s^n(k)v_s^n(k)}\,dB^{1,n}_s(k),\\
			v_t^n(k)&=v_0(k)+\int_0^t \sum_{\underset{|j-k|=1}{ j\in S_n}}\frac{1}{2d}(v^n(j)-v^n(k))\,ds+\int_0^t\sqrt{\gamma u_s^n(k)v_s^n(k)}\,dB^{2,n}_s(k).
		\end{align*}
		The correlation structure of the Brownian motions remains as in Definition \ref{defsol}. Since this is a system of finite-dimensional stochastic differential equations existence of weak solutions 
		$$\big\{u^n_t(k),v^n_t(k),B^{1,n}(k),B^{2,n}(k)\big\}_{k\in S_n}$$ follows from finite-dimensional diffusion theory for sufficiently ``good" coefficients (see for instance Theorem 5.3.10 of \cite{EK86}). To prove non-negativity of solutions, one shows that the semimartingale's local time at zero equals to zero (see for instance page 1127 of \cite{DP98}).

		Solutions $(u^n,v^n)$ can be extended to the entire lattice by setting $u^n_t(k)=u_0(k),v^n_t(k)=v_0(k)$ for $k\neq S_n$. Due to the choice of the initial conditions, the $(u^n_t, v^n_t)$ are contained in $L_\beta^2$ for all $t\geq 0$.

		The main ingredients, to prove convergence of $(u^n,v^n)$, are the following estimates. It suffices to show that for $k\in S, T>0$, and $\epsilon>0$
		\begin{align}
			\sup_{n\in\N}\P\big[\sup_{t\leq T}u^n_t(k)>K\big]&\rightarrow 0, \quad \text{ as }K\rightarrow \infty,\label{e1}\\
			\sup_{n\in\N}\sup_{|t-s|\leq h, 0\leq t,s\leq T}\P\big[|u^n_t(k)-u_s^n(k)|>\epsilon\big]&\rightarrow 0, \quad\text{ as }h\rightarrow 0,\label{e2}
		\end{align}
		and analogously for $v^n$. The desired convergence in (\ref{e1}), (\ref{e2}) is analogous to (2.9) and (2.10) of \cite{SS80}. In order to ensure that all stochastic integrals are martingales we introduce a sequence of stopping times: $T^n_N=\inf\big\{t\geq 0: \langle u_t^n,\beta\rangle+\langle v_t^n,\beta\rangle>N\big\}$. This sequence, almost surely, converges to infinity, as $N$ tends to infinity, since solutions do not explode. Using only the definition of $\textrm{SBM}^n_\gamma$ we estimate
		\begin{align}
\nonumber
			\quad\E\big[\sup_{t\leq T\wedge T_N^n}\langle u_t^n,\beta\rangle\big]
			&=\E\Big[\sup_{t\leq T\wedge T_N^n}\sum_{i\in \Z^d}u_t^n(i)\beta(i)\Big]\\
\nonumber
			&\leq\langle u_0,\beta\rangle+\E\bigg[\sup_{t\leq T\wedge T_N^n}\sum_{i\in S_n}\beta(i)\int_0^t \sum_{\underset{|j-i|=1}{j\in S_n}}\frac{1}{2d}(u^n_s(j)-u^n_s(i)) \,ds\bigg]\\
\nonumber
			&\quad+\E\bigg[\sup_{t\leq T\wedge T_N^n}\sum_{i\in S_n}\beta(i)\int_0^t\sqrt{\gamma u_s^n(i)v_s^n(i)}\,dB^{1,n}_s(i)\bigg]\\\begin{split}
			&\leq\langle u_0,\beta\rangle+\E\bigg[\sum_{i\in S_n}\beta(i)\int_0^{T\wedge T_N^n} \sum_{\underset{|i-j|=1}{j\in S_n}}\frac{1}{2d}u_s^n(j)\,ds\bigg]
\\
\label{2807_1}
&\quad +\E\bigg[\sup_{t\leq T\wedge T_N^n}\sum_{i\in S_n}\beta(i)\int_0^t\sqrt{\gamma u_s^n(i)v_s^n(i)}\,dB^{1,n}_s(i)\bigg].\end{split}
		\end{align}
		Using the Burkholder-Davis-Gundy inequality and then Fubini's theorem we obtain the 
 following upper bound for the above expressions
		\begin{eqnarray*}
			\langle u_0,\beta\rangle+\sum_{i\in S_n}\beta(i)\int_0^{T} \sum_{\underset{|i-j|=1}{j\in S_n}}\frac{1}{2d}\E\big[u_s^n(j)\big]\,ds+\gamma\sum_{i\in S_n}\beta(i)\int_0^{T}\E\big[ u_s^n(i)v_s^n(i)\big]\,ds.
		\end{eqnarray*}
		So far, this procedure is fairly standard for interacting diffusions of type (\ref{int}) where instead of the mixed moments, the expectations $\E[f(w_t(i))]$ need to be bounded. There, linear growth conditions on $f$ lead to a Gronwall inequality which yields the desired bound. In our case, we need to estimate moments $\E[u_s^n(j)]$ and $\E[u_s^n(i)v_s^n(i)]$ uniformly in $n$. The first moment can be estimated as on page 1129 of \cite{DP98} since the correlations do not influence the first moments. The mixed second moment is more delicate. Using a point wise representation of solutions, for $\varrho<0$, the same estimates as in \cite{DP98} can be performed. The additional difficulty comes for positively correlated Brownian motions ($\varrho>0$) that spoil the Gronwall argument in~\cite{DP98} due to the appearance of an additional positive summand. Nonetheless, the mixed second moment for the approximating system can be estimated directly: a moment expression for the finite-
 dimensional equation, in the same spirit of the moment duality that we explain below (see Lemma \ref{la:mdual}),  gives the uniform in $n$ upper bound 
		\begin{align}\label{abc}
			\E\big[u_s^n(i)v_s^n(i)\big]\leq C e^{\gamma T}.
		\end{align}
		This is similar to the remark on page 41 of \cite{CDG04} where the existence of solutions for $\varrho=0$ was justified by the observation that $uv\leq u^2+v^2$ which leads to an upper bound by a two-type Anderson model verifying (\ref{abc}).\\
		By monotone convergence, getting rid of the stopping times on the lefthand side of~(\ref{2807_1}), this implies
		\begin{eqnarray*}
			\E\big[\sup_{t\leq T}\langle u_t^n,\beta\rangle\big]=\lim_{N\rightarrow \infty}\E\big[\sup_{t\leq T\wedge T_N^n}\langle u_t^n,\beta\rangle\big]
			\leq \langle u_0,\beta \rangle+C_T,
		\end{eqnarray*}
		where $C_T$ is independent of $n$. Hence, in particular we get by Chebychev's inequality that
		\begin{eqnarray*}
		 	\sup_{n\in\N}\P[\sup_{t\leq T}u^n_t(k)>K]\rightarrow 0,
		\end{eqnarray*}
		as $K$ tends to infinity. To prove (\ref{e2}) one needs to check that $\E\big[|\langle u_t^n-u_s^n,\beta\rangle|\big]$ is bounded uniformly in $n$. This can be done similarly as before, using the same bounds on the 				moments.\\
		Following the arguments on page 399 of \cite{SS80}, the bounds (\ref{e1}), (\ref{e2}) suffice to ensure convergence (in a sufficiently strong sense) of the sequences $(u^n,v^n)$ to a limiting process $(u,v)$ solving the equation defining $\textrm{SBM}_\gamma$.
\end{proof}

\subsubsection{Tools: Mild Solutions, Total-Mass Processes and Dualities}
	From the very definition, interacting diffusion processes are parabolic equations with random potential functions. In the spirit of the deterministic theory one can equally ask for representations that are easier to work with in some 		situations. We will use the weak-solution representation and  the variation of constant form. In the following we use the semigroup generated by $\Delta$ on $\Z^d$, i.e. the family of linear operators
	\begin{align}\label{SG}
		P_t f(k)=\sum_{i\in \Z^d}p_t(i,k)f(i),\quad t\geq 0,
	\end{align}
	where $p_t(i,k)$ is the transition kernel of a simple random walk on $\Z^d$.
	\begin{convention}
		The constant function on $\Z^d$ taking value $u\in \R^+$ is abbreviated by $\bu$.
	\end{convention}
	For a continuum analogue of the following two representations we refer to Corollary 19 of \cite{EF04} and for a very detailed proof on the lattice for $\varrho=0$ to Theorem 2.2 of \cite{DP98}.
	\begin{prop}\label{prop:mild}
		Suppose that $(u_t,v_t)$ is a solution of $\textrm{SBM}_\gamma$ with $u_0,v_0$ summable, then $u_t$ and $v_t$ are summable and the total-mass processes satisfy
		\begin{align}\label{12}
		\begin{split}
			\langle u_t,\1\rangle&=\langle u_0,\1\rangle+\sum_{j\in \Z^d}\int_0^t \sqrt{\gamma u_s(j)v_s(j)}\,dB^1_s(j),\\
			\langle v_t,\1\rangle&=\langle v_0,\1\rangle+\sum_{j\in \Z^d}\int_0^t \sqrt{\gamma u_s(j)v_s(j)}\,dB^2_s(j),
		\end{split}
		\end{align}
		where the infinite sums converge in $L^2(\P)$. If $(u_0,v_0)\in L^2_\beta$, then the point wise representation
		\begin{align}
		\begin{split}\label{14}
			u_t(k)&=P_tu_0(k)+\sum_{j\in\Z^d}\int_0^tp_{t-s}(j,k)\sqrt{\gamma u_s(j)v_s(j)}\,dB^1_s(j),\\
			v_t(k)&=P_tv_0(k)+\sum_{j\in\Z^d}\int_0^tp_{t-s}(j,k)\sqrt{\gamma u_s(j)v_s(j)}\,dB^2_s(j),
		\end{split}
		\end{align}
		holds. The covariation structure of the Brownian motions is as in  the definition of $\textrm{SBM}_\gamma$.
	\end{prop}
	The weak solution representation can be obtained for $\langle u_t,\phi\rangle$ also for more general test-functions $\phi$.
 Instead of sketching a proof we give an important application leading the way from $\textrm{SBM}_\gamma$ to the exit-law $Q^\varrho$ defined in (\ref{exitd}).
	\smallskip

	A property that is shared by many particle systems is that started at summable initial conditions the total-mass process is a martingale. A natural question for the two-type model $\textrm{SBM}_\gamma$ is how the two total-mass martingales relate to each other if both types are started at summable initial conditions. In order to avoid confusion with $\langle \cdot,\cdot\rangle$ we denote in the following the cross-variations of square-integrable martingales by $[\cdot,\cdot]$.

	\begin{lem}\label{lem:3}
		Suppose $u_0$ and $v_0$ are summable, then $\langle u_t,\1\rangle$ and $\langle v_t,\1\rangle$ are non-negative square-integrable martingales with quadratic-variations
		\begin{align*}
			\left[\langle u_\cdot,\1\rangle,\langle u_\cdot,\1\rangle\right]_t=\left[\langle v_\cdot,\1\rangle,\langle v_\cdot,\1\rangle\right]_t=\gamma \int_0^t \langle u_s,v_s	\rangle\,ds
		\end{align*}
		and cross-variation
		\begin{align*}
			\left[\langle u_\cdot,\1\rangle,\langle v_\cdot,\1\rangle\right]_t=\varrho\gamma \int_0^t \langle u_s,v_s\rangle\,ds.
		\end{align*}	
	\end{lem}
	\begin{proof}[Sketch of Proof]
		Positivity of the total-mass processes comes directly from positivity of symbiotic branching processes and the martingale property follows from the representation in Proposition \ref{prop:mild} as the martingale property 		is invariant under $L^2(\P)$-convergence. Hence, it suffices to calculate the bracket-processes. The representation stems from the fact that for $L^2(\P)$-convergent martingales also the bracket processes converge so that
		\begin{align*}
			[\langle u_\cdot,1\rangle,\langle v_\cdot,1\rangle]_t&=\lim_{|M|\to\infty}\Big[\sum_{k\in M}\int_0^\cdot \sqrt{\gamma u_s(k)v_s(k)}\,dB^1_s(k),\sum_{k\in M}\int_0^\cdot \sqrt{\gamma u_s(k)v_s(k)}\,dB^2_s(k)\Big]_t\\
			&=\varrho\lim_{|M|\to\infty}\sum_{k\in M}\int_0^t \gamma u_s(k)v_s(k)\,ds\\
			&=\varrho\gamma \int_0^t \langle u_s,v_s\rangle\,ds.
		\end{align*}
		The derivation of the quadratic variations is similar but without the additional correlation parameter $\varrho$.
	\end{proof}

\subsubsection{Dualities}
	The results on interacting particle systems obtained during the last decades showed that the depth of possible results for particular systems depends 
  in many cases on available duality relations, i.e. relations of 			characteristics (here: Laplace transforms or  moments) of the process to those of other processes. \\
	In general, one says that a duality between two Markov processes $X$ and $Y$ with state-spaces $\mathcal X$ and $\mathcal Y$ holds if for 	some duality-function $H:\mathcal X\times \mathcal Y\to \R$ and a potential function $f:\mathcal Y\to \R$
	\begin{align}\label{dual}
		\E^{X_0}[H(X_t,Y_0))]=\E^{Y_0}\big[H(X_0,Y_t)e^{\int_0^t f(Y_s)\,ds}\big],\quad t\geq 0.
	\end{align}
	In fact, the simplified definition of duality involves potential function $f=0$ but we will need the generalized formulation for one of the dualities for $\textrm{SBM}_\gamma$.	Such a semigroup relation can also be expressed via generators. Suppose $A^X$ is the generator for $X$ and $A^Y$ the generator for $Y$, then, under some conditions (see Corollary 4.4.13 of \cite{EK86}), (\ref{dual}) is equivalent to the generator identity
	\begin{align}\label{gen}
		A^X H(\cdot ,y)=A^Y H(x,\cdot)+f(\cdot)H(x,\cdot),\quad \forall x\in \mathcal X, \forall y\in \mathcal Y.
	\end{align}
	The exponential correction is caused by the Feynman-Kac representation for the semigroup with additional potential $f$. The generator relation turns out to be useful in many cases to find a duality expression.\\
	 Let us now have a look at the two dual relations for $\textrm{SBM}_\gamma$. For $(x_1,x_2), (y_1,y_2)\in L_\beta^{2}$, define 
	\begin{align}\label{sd}\begin{split}
		&\quad\langle\langle x_1,x_2,y_1,y_2\rangle\rangle_\varrho\\
		&= \sum_{k\in \Z^d}\Bigg[-\sqrt{1-\varrho}\big( x_1(k)+x_2(k)\big)\big(y_1(k)+y_2(k)\big)+i\sqrt{1+\varrho}\big( x_1(k)-x_2(k)\big)\big(y_1(k)-y_2(k)\big)\Bigg].\end{split}
	\end{align}
Then, the self-duality relation reads as follows: suppose $(u_t,v_t)$ and $(\tilde u_t,\tilde v_t)$ are two solutions of $\textrm{SBM}_\gamma$ with initial conditions $(u_0,v_0)\in L_\beta^2$ and $\tilde u_0,\tilde v_0$ with compact support, then 
	\begin{align}\label{la:selfdual}
		\E^{u_0,v_0}[\exp(\langle\langle u_t,v_t,\tilde u_0,\tilde v_0 \rangle\rangle_\varrho)]=\E^{\tilde u_0,\tilde v_0}[\exp(\langle\langle u_0,v_0,\tilde u_t,\tilde v_t \rangle\rangle_\varrho)].
	\end{align}
	For $\varrho=0$ the self-duality goes back to \cite{M98} (see also Section 4 of \cite{CDG04}) and was generalized subsequently to $\varrho\neq 0$ in \cite{EF04}. In order to define the infinite rate analogue $\textrm{SBM}_\infty$, the self-duality will be 	discussed in more detail in Section 3. On first view the duality looks frightening but it has very important applications. Here is an application to weak uniqueness (Proposition 5 of \cite{EF04}) of $\textrm{SBM}_\gamma$.
	\begin{lem}\label{uniq}
		Weak uniqueness holds for solutions to $\textrm{SBM}_\gamma$ if $\varrho\in (-1,1)$.
	\end{lem}
	\begin{proof}[Sketch of Proof]
  In fact, duality in many cases implies weak uniqueness for corresponding processes: for a general result see Proposition~4.4.7 in~\cite{EK86}. In this particular case, 
		the proof goes roughly as follows: suppose there are two solutions $(u^1_t,v^1_t)$ and $(u^2_t,v^2_t)$ with identical initial condition $(u_0,v_0)$. We aim at showing that both solutions coincide in law. For any fixed pair of 		compactly supported sequences $\phi,\psi$ there is a solution $(\tilde u_t,\tilde v_t)$ of $\textrm{SBM}_\gamma$ with initial condition $(\phi,\psi)$. Applying the self-duality twice shows that for all $t\geq 0$
		\begin{align*}
			\E^{u_0,v_0}[\exp(\langle\langle u^1_t,v^1_t,\phi,\psi\rangle\rangle_\varrho)]
			&=\E^{\phi,\psi}[\exp(\langle\langle u_0,v_0,\tilde u_t,\tilde v_t \rangle\rangle_\varrho)]\\
			&=\E^{u_0,v_0}[\exp(\langle\langle u^2_t,v^2_t,\phi,\psi \rangle\rangle_\varrho)].
		\end{align*}
		A closer look at the complex-valued duality function shows that the equality suffices to deduce the claim: in the real direction this is a Laplace transform and in the imaginary direction a Fourier transform. Hence, it comes as 		no surprise that the Laplace transform uniquely determines the law of $\langle u^i_t+v^i_t,\phi+\psi\rangle $ and the Fourier transform uniquely determines the law of $\langle u^i_t-v^i_t,\phi-\psi \rangle$. Now, as for finite-dimensional random 		variables, since $\phi$ and $\psi$ are arbitrary, this uniquely determines the laws of the configurations $u^i_t+v^i_t$ and $u^i_t-v^i_t$. Taking sums and differences, the one-dimensional marginals of $(u^i,v^i)$ are 			determined which finally can be extended to the path-level by Markov process arguments.
	\end{proof}
	The second duality is of different type. It does not involve an exponential duality function but a polynomial only. The aim is to give an expression for the moments
	\begin{align*}
	 	\E[u_t(k_1)\cdots u_t(k_{n})v_t(k_{n+1})\cdots v_t(k_{n+m})]
	\end{align*}
	in terms of a particle system which we now describe. Suppose that $n+m$ particles in $\Z^d$ are given. Each particle moves according to a continuous-time simple random walk independently of all other particles. At time $0$, $n$ particles of color $1$ are located at positions $k_1,...,k_n$ and $m$ particles of color $2$ are located at positions $k_{n+1},...,k_{n+m}$. For 	each pair of particles having the same color, one particle of the pair changes its color when the time the two particles have spent at same site, exceeds an independent exponential time with parameter $\gamma$. Let
	\begin{align*}
		L_t^=&=\text{total collision time of all pairs of same colors up to time }t,\\
		L_t^{\neq}&=\text{total collision time of all pairs of different colors up to time }t,\\
		l^1_t(a)&=\text{number of particles of color }1\text{ at site }a\text{ at time t},\\
		l^2_t(a)&=\text{number of particles of color }2\text{ at site }a\text{ at time t}
	\end{align*}
	and define the duality function
	\begin{align*}
		(x_1,x_2)^{(A_1,A_2)}&=\prod_{k\in \Z^d}x_1(k)^{A_1(k)}\prod_{k\in \Z^d}x_2(k)^{A_2(k)}
	\end{align*}
	for $(x_1,x_2)\in L_\beta^2$ and $A_1,A_2\in \{a:\Z^d\to\N \,|\, a(k)\neq 0\text{ only finitely many times} \}$. By definition the number of particles is finite and constant so that the duality function can be applied to $l^1, l^2$.
	The following lemma is taken from Section 3 of \cite{EF04}.
	\begin{lem}\label{la:mdual}
		Suppose $(u_0,v_0)\in L_\beta^2$ and $k_i\in \Z^d$, then
		\begin{align*}
			\E[u_t(k_1)\cdots u_t(k_{n})v_t(k_{n+1})\cdots v_t(k_{n+m})]=\E\big[(u_t,v_t)^{(l_0^1,l_0^2)}\big]&=\E\big[(u_0,v_0)^{(l^1_t,l^2_t)}e^{\gamma(L_t^=+\varrho L_t^{\neq})}\big],
		\end{align*}
		where the dual process behaves as explained above.
	\end{lem}
	In the general setting $\mathcal A\neq \Delta$, only the dynamics of the single particles needs to be changed: they move as continuous-time Markov process with generator $\mathcal A$ and perform the same changes of 		colors. \\
	A simpler situation occurs if the initial conditions are homogeneous. If $u_0=v_0=\1$, then the righthand side of the duality only consists of the exponential perturbation.
	\begin{remark}
		 In the special case of $\varrho=1$ and $u_0=v_0=\1$, Lemma \ref{la:mdual} was already stated in \cite{CM}, reproved in \cite{GdH07}, and used to analyze the Lyapunov exponents of the parabolic Anderson model. Since 		then the collision times $L^=_t$ and $L_t^{\neq}$ are weighted equally, the colors can be ignored so that only exponential moments of collision times need to be analyzed.
	\end{remark}

	For $\varrho\neq 1$, the difficulty of the dual expression is based on the two interacting stochastic effects: on the one hand, one has to deal with collision times of random walks which were analyzed in \cite{GdH07} and 		additionally particles have colors either $1$ or $2$ which change dynamically.
		
\subsection{Longtime Behavior and the Exit-Law of $\varrho$-correlated Brownian Motions}\label{sec:longtime}
	Now, after we discussed  the basic properties and tools, we can derive the connections of $\textrm{SBM}_\gamma$ and the exit-law $Q^\varrho$. This will explain some longtime properties of $\textrm{SBM}_{\gamma}$. For $\varrho=0$ these properties go back to \cite{DP98} 	and \cite{CK00}, and they were studied  for $\varrho\in (-1,1)$	in \cite{BDE11} and  
 for $\varrho\in (-1,0)$ in~\cite{DM11}.	
	In the following we briefly explain the main results on the longtime behavior and sketch the ideas of the proofs. 
  Most importantly, we show how the exit law (\ref{exitd}) can be related to 
	$\lim_{t\to\infty} \langle u_t,\1\rangle$, via Lemma \ref{lem:3},  and then this has several interesting consequences:  the weak limit	
		$\lim_{t\to\infty} \mathcal L (u_t,v_t)$
	can be deduced for infinite initial conditions unifying two classical results for the stepping stone model and the parabolic Anderson model with Brownian potential. From this, a general technique of \cite{CK00} can be applied 	to deduce non-convergence in the almost sure sense of the $L_\beta^2$-valued processes $(u_t,v_t)$.
	 As a final consequence of the connection to $Q^\varrho$, we discuss how a critical curve for the behavior of higher moments
	can be deduced which then leads us to the definition of $\nu^\varrho$ (from Theorem~\ref{0}). 

\subsubsection{Longtime Analysis for the Total Mass - Transience/Recurrence-Dichotomy}\label{longl}
	A typical question for interacting particle systems is the following: does the system get extinct eventually if the initial conditions are summable? For several examples this question can be answered easily by taking 	into account nice duality structures. For instance, as explained in the introduction, for the voter process this question is equivalent to finite time coalescence of finitely many random walks which again is equivalent to recurrence 	of the random walks.\\
	The question is more subtle for $\textrm{SBM}_\gamma$ by lack of knowledge of a useful duality; the self-duality and the moment-duality are not very helpful here.
 Nonetheless, a weaker question can still be			answered by other arguments.
	Also note that instead of extinction/non-extinction, the question of coexistence/non-coexistence is addressed 
for two type symbiotic model. Let us first define a notion of ``coexistence" in the two-type $\textrm{SBM}_{\gamma}$ model for which we utilize the martingale property of the non-negative total-mass processes (recall Lemma \ref{lem:3}). The martingale convergence theorem implies existence of
		\begin{align*}
		\lim_{t\to\infty}\langle u_t,\1\rangle\langle v_t,\1\rangle=:\langle u_\infty,\1\rangle\langle v_\infty,\1\rangle \in [0,\infty)
	\end{align*}
	leading to the following definition.
	\begin{definition}\label{def:coex}
		We say that coexistence (of types) is possible, if there are summable initial conditions $u_0,v_0$ such that  $\langle u_\infty,\1\rangle\langle v_\infty,\1\rangle>0$ with positive probability. Otherwise we say that coexistence is impossible.
	\end{definition}
	
	For $\textrm{MCB}_\gamma$, Dawson/Perkins proved the following recurrence/transience dichotomy. (Recall again that we state everything for the case of $S=\Z^d$ and $\mathcal A=\Delta$, while some of the results have been proved in a more general setting.)
	\begin{thm}[Dichotomy for Finite Initial Conditions]\label{thm:dp}
		Coexistence of types for $\textrm{MCB}_\gamma$ is possible if and only if a simple random walk on $\Z^d$ is transient (i.e. $d\geq 3$).
	\end{thm}
	
	For $\varrho\neq 0$ the situation is not completely clear, yet.
The following result was proved in \cite{BDE11}.
\begin{prop}
\label{prop:100}
	  Let  $\varrho\neq 0$, then coexistence of types for $\textrm{SBM}_\gamma(\varrho)$  
  is impossible if a simple random walk on $\Z^d$ is recurrent (i.e. $d=1,2$).
	\end{prop}
 For $\varrho<0$ the above result has been improved in~\cite{DM11} and the following recurrence/transience dichotomy 
has been verified there via moment arguments.
	\begin{prop}
\label{prop:101}
		Let $\varrho<0$, then coexistence of types for $\textrm{SBM}_\gamma(\varrho)$ is possible if and only if a simple random walk on $\Z^d$ is transient (i.e. $d\geq 3$).
	\end{prop}
	Propositions~\ref{prop:100} and \ref{prop:101} imply that the question that remains open is whether coexistence of types is possible whenever  $\varrho>0$ and a simple random walk on $\Z^d$ is transient (i.e. $d\geq 3$).
	\begin{conj}
		We conjecture that if $d\geq 3$ and $\varrho>0$, there is a critical constant $\gamma(\varrho,d)\in (0,\infty)$ such that coexistence of types for $\textrm{SBM}_\gamma(\varrho)$ occurs if $\gamma<\gamma(\varrho,d)$ and coexistence is impossible if $\gamma>\gamma(\varrho,d)$.
	\end{conj}
	The conjecture is based on a known similar statement for the parabolic Anderson model corresponding to $\varrho=1$ (see \cite{GdH07}).
	\medskip
	
	In the following we are going to sketch the arguments used in the approach of \cite{DP98} to prove Theorem \ref{thm:dp}.
	\begin{lem}\label{la:eds}
		Suppose $u_0$ and $v_0$ are summable and let
			$T(t)=\gamma \int_0^t \langle u_s,v_s\rangle\,ds.$
		Then
		\begin{align*}
	 		(W^1_t,W^2_t):=\big(\langle u_{T^{-1}(t)},\1\rangle,\langle v_{T^{-1}(t)},\1\rangle\big)
		\end{align*}
 		is a pair of $\varrho$-correlated Brownian motions stopped when it hits the boundary $E$ of the first quadrant.
	\end{lem}
	\begin{proof}[Sketch of Proof]
		The claim follows directly from Lemma \ref{lem:3} and the Dubins-Schwartz theorem (see for instance Theorem 3.4.6 of \cite{KS98}) applied to both total-mass processes separately. The correlation is directly inherited from the driving noises.
	\end{proof}
	
	With the lemma in hand, let us emphasize the idea behind Theorem \ref{thm:dp}.
	The pair of total-masses is a time-changed planar Brownian motion that stops once it hits $E$. Hence, in order to prove the 		theorem one has to find a characterization under which the time-change $T^{-1}$ levels off 		before the planar Brownian motion hits $E$, i.e. one has to show
	\begin{align*}
		\langle u_\infty,\1\rangle\langle v_\infty,\1\rangle=W^1_\tau W^2_\tau=0\quad &\Longleftrightarrow\quad  T(\infty)=\gamma \int_0^\infty \langle u_s,v_s\rangle\,ds=\tau\quad \Longleftrightarrow \quad d=1,2.
	\end{align*}
	Morally, in the transient case islands carrying only type $u$ and islands carrying only type $v$ can move away from each other so that $\langle u_s,v_s\rangle$ is small even though $\langle u_s,\1\rangle, \langle v_s,\1\rangle$ 	are not. Such scenarios might cause the time-change to stop increasing before one of the types vanished.
	\smallskip
	
\begin{remark}
\label{rem:2807_1}
In fact, the argument of Dawson/Perkins  
 proves more, also for $\varrho\in(-1,1)$, than stated in Theorem \ref{thm:dp}. For $d=1,2$ and
 $\varrho\in(-1,1)$, the limit $(\langle u_\infty,\1\rangle,\langle v_\infty,\1\rangle)$ is distributed according to the law $Q^\varrho_{\langle u_0,\1\rangle,\langle v_0,\1\rangle}$ of  $(W^1_\tau, W^2_\tau)$ (see \cite{BDE11}) and this will be the crucial ingredient of the proof of the next theorem.
\end{remark}

\subsubsection{Weak Longtime Analysis for Infinite Initial Conditions in dimensions $d=1,2$}  
	
	An important observation of Dawson/Perkins in their study of $\textrm{MCB}_\gamma$ was the use of the self-duality to deduce from Theorem \ref{thm:dp} also the longtime behavior for the system started at infinite initial 	conditions. 	This task is often much more complicated than in the case of summable initial conditions 
	 since simple martingale arguments fail. The extension of their approach to $\varrho\neq 0$ leads to the next result for the recurrent regime (see Proposition 2.1 of \cite{BDE11}). The Brownian motions $W^1,W^2$ and the stopping 	time $\tau$ are as in (\ref{exitd}) and (\ref{exitt}).
	\begin{thm}[Infinite Initial Conditions]\label{prop:convlaw}
		Suppose that $d\leq 2$, $\varrho\in (-1,1)$ and $u_0=\bu, v_0=\bv$, then
		\begin{align*}
			\mathcal L^{\bu,\bv}  (u_t,v_t)\stackrel{t\to\infty}{\Longrightarrow} \mathcal L^{u,v} (\bar W^1_\tau,\bar W^2_\tau).
		\end{align*}
		Here, $(\bar W^1_{\tau},\bar W^2_{\tau})$ denotes the pair of functions with constant values $( W^1_{\tau},W^2_{\tau})$.
	\end{thm}
	\begin{proof}[Sketch of Proof]
		Due to the arguments in the uniqueness proof of Corollary \ref{uniq}, it comes as no suprise that 
		\begin{align*}
			&\quad \mathcal L^{\bu,\bv} (u_t,v_t) \text{ converges weakly to }\mathcal L^{u,v} (\bar W^1_{\tau},\bar W^2_{\tau})\\
			&\Longleftrightarrow \quad\lim_{t\to\infty}\E^{\bu,\bv}[\exp(\langle\langle u_t,v_t,\phi,\psi\rangle\rangle_\varrho)]=\E^{u,v}[\exp(\langle\langle \bar W^1_\tau,\bar W^2_\tau,\phi,\psi\rangle\rangle_\varrho)],\quad \forall\text{ compactly supported } \phi,\psi.
		\end{align*}
		The argument is the same: convergence of the Laplace transform determines weak convergence of the sum $u_t+v_t$ and convergence of the Fourier transform determines convergence of the difference $u_t-v_t$. Taking 		sums and differences, convergence of the pair $(u_t,v_t)$ follows. Employing the self-duality, one obtains by dominated convergence 
		\begin{align*}
			&\quad\lim_{t\to\infty}\E^{\bu,\bv}[\exp(\langle\langle u_t,v_t,\phi,\psi\rangle\rangle_\varrho)]\\
			&=\lim_{t\to\infty}\E^{\phi,\psi}[\exp(\langle\langle \bu,\bv,\tilde u_t,\tilde v_t\rangle\rangle_\varrho)]\\
			&=\E^{\phi,\psi}\left[\exp\left(-\sqrt{1-\varrho}(u+v)\lim_{t\to\infty}\langle \1,\tilde u_t+\tilde v_t\rangle+i\sqrt{1+\varrho}(u-v)\lim_{t\to\infty}\langle \1,\tilde u_t-\tilde  v_t\rangle\right)\right]
		\end{align*}
		so that the almost sure convergence of the total-mass processes $(\langle \1,\tilde u_t\rangle, \langle \1,\tilde v_t\rangle)\to (\langle \1,\tilde u_\infty\rangle, \langle \1,\tilde v_\infty\rangle)\stackrel{d}{=}(W^1_\tau, W^2_\tau)$ (see Remark~\ref{rem:2807_1})
gives equality to
		\begin{align*}
			E^{\langle \phi,\1\rangle,\langle \psi,\1\rangle}\left[\exp\left(-\sqrt{1-\varrho}(u+v)(W^1_\tau+W^2_\tau)+i\sqrt{1+\varrho}(u-v)( W^1_\tau-W^2_\tau)\right)\right].
		\end{align*}
 		Finally, it remains to show the identity
		\begin{align*}
			&\quad	E^{\langle \phi,\1\rangle,\langle \psi,\1\rangle}\left[\exp\left(-\sqrt{1-\varrho}(u+v)(W^1_\tau+W^2_\tau)+i\sqrt{1+\varrho}(u-v)( W^1_\tau-W^2_\tau)\right)\right]\\
			&=E^{u,v}[\exp(\langle\langle \bar W^1_\tau,\bar W^2_\tau,\phi,\psi\rangle\rangle_\varrho)]
		\end{align*}
		for $\varrho$-correlated Brownian motions which can either be done via stochastic calculus (for $\varrho=0$ see page 1111 of \cite{DP98}) or by considering the self-duality for the simplest symbiotic branching model
		\begin{align*}
		\begin{cases}
			du_t=\sqrt{\gamma u_tv_t}dB^2_t\\
			dv_t=\sqrt{\gamma u_tv_t}dB^2_t
		\end{cases}
		\end{align*}
		(see the proof of Proposition 4.4 of \cite{BDE11}).
	\end{proof}
	
	The need for the assumption $d\leq 2$ becomes apparent in the proof: due to the self-duality, the convergence claimed in the theorem is equivalent to convergence of laws of the total-mass processes to $Q^\varrho$. This  
occurs precisely in dimensions  $d=1,2$ (see Remark~\ref{rem:2807_1}).	
	\medskip
	
	The most striking feature of the extension in Theorem \ref{prop:convlaw} from $\varrho=0$ to $\varrho\in (-1,1)$ is a unifying property for $\textrm{SBM}_\gamma$ in the recurrent regime based on the unifying property 			presented in the Section \ref{sec:inter}. The theorem is restricted to $\varrho\in (-1,1)$ 
which is partly caused by the use of the self-duality in the proof. However, comparing Theorem~\ref{prop:convlaw} with well-known results for the boundary cases $\varrho=-1,1$, in dimensions $d=1,2$, it turned out that the classical results can be reformulated in a unified language via $\varrho$-correlated Brownian motions.
	\smallskip
	
	First, suppose $w_t$ is a solution of the stepping stone model (see Example \ref{ex1}), 
 in dimension $d\leq 2$,  and $w_0\equiv w \in [0,1]$. It was proved in \cite{S80} that
\begin{eqnarray}
\label{eq:shigalimit}
	\mathcal{L}^{\bf w}(w_t)\stackrel{t\rightarrow \infty}{\Longrightarrow} w \delta_{\1}
         +(1-w) \delta_{\bf 0},
\end{eqnarray}
where $\delta_{\1}$ (resp.~$\delta_{\bf 0}$) denotes the Dirac distribution concentrated on the constant function $\1$ (resp. $\bf 0$). This can be reformulated in terms of perfectly anti-correlated Brownian motions $(W^1, W^2)$: For $\varrho=-1$, the pair $(W^1, W^2)$ takes values only on the straight-line connecting $(0,1)$ and $(1,0)$, and stops at the boundaries. Hence, the law of $(W^1_{\tau}, W^2_{\tau})$ is a mixture of $\delta_{(0,1)}$ and $\delta_{(1,0)}$ and the probability of hitting $(1,0)$ is equal to the probability of a one-dimensional Brownian motion started at $w \in [0,1]$ hitting $1$ before $0$, which is $w$, and hence it matches (\ref{eq:shigalimit}). \\
Secondly, let $w_t$ be a solution of the parabolic Anderson model with Brownian potential (see Example \ref{pam}), in dimension $d\leq 2$, and constant initial condition $w_0 \equiv w \geq 0$. In \cite{r7} it was shown that 
\begin{eqnarray*}
	\mathcal{L}^{\bf w}(w_t)\stackrel{t\rightarrow \infty}{\Longrightarrow} \delta_{\bf 0}.
\end{eqnarray*}  
As discussed above, if the Anderson model is viewed as a symbiotic branching process with $\varrho=1$, this implies
\begin{eqnarray*}
	\mathcal{L}^{\bf w, {\bf w}}(u_t,v_t)\stackrel{t\rightarrow \infty}{\Longrightarrow} \delta_{\bf 0, 0}.
\end{eqnarray*}
From the viewpoint of two perfectly positive-correlated Brownian motions in Theorem \ref{prop:convlaw} we obtain the same result since they simply move on the diagonal dissecting the upper right quadrant until they eventually get absorbed at the origin, i.e. $(W^1_\tau, W^2_\tau)=(0,0)$ almost surely.

\subsubsection{Almost Sure Longtime Analysis for $d=1,2$}

	Started in homogeneous initial conditions, Theorem \ref{thm:dp} states that for $d=1,2$, $(u_t,v_t)$ converges weakly to a law under which one type completely disappears. It is natural to ask whether this convergence also 		holds pathwise. The negative answer is the following result taken from \cite{BDE11} which is based upon the general strategy of \cite{CK00}.
	
	\begin{thm}[Almost Sure Non-Convergence]\label{prop:pb}
		Suppose $(u_t,v_t)$ solves $\textrm{SBM}_\gamma$ in dimensions $d=1,2$ for $\varrho\in (-1,1)$ with $u_0=\bu, v_0=\bv$. Then
		\begin{eqnarray*}
			\P \left[ \liminf_{t\rightarrow \infty} \sup_{k\in K} \left|\left(u_t(k) \atop v_t(k)\right)-\left(u' \atop v'\right)\right|=0 \right] =1
		\end{eqnarray*}
		for all $(u',v')\in E$ and $K \subset \Z^d$ bounded.
	\end{thm}

	Contrasting the weak convergence to $\mathcal L \big(\bar W^1_\tau, \bar W^2_\tau\big)$ of the previous section, the almost sure behavior is entirely different: In contrast to choosing,  according  to the exit-law $Q^{\varrho}$,  \textbf{one} point $(u',v')\in E$ as a	limit 	point, the process locally approaches \textbf{every} possible $(u',v')\in E$ infinitely often. Hence, looking at a fixed box $K$, the dominant type changes infinitely often and both types approach arbitrarily high values.
\smallskip

		Theorem \ref{prop:pb} is consistent with the known results for the boundary cases even though they appear to be very different. Looking inside the proofs of \cite{CK00},
one realizes that, in fact, they proved more than we stated here. Each point of the 		support of the limit measure $\big(\bar W^1_\tau,\bar W^2_\tau\big)$ is an accumulation point for $(u_t,v_t)$. Plugging this into the limit measures of the boundary cases discussed below Theorem \ref{prop:convlaw}, the theorem extends smoothly to $\varrho=-1$: The support of $u\delta_{\1}+(1-u)\delta_{\mathbf{0}}$ only contains $\1$ and $\mathbf{0}$ leading to the classical result that solutions to Example \ref{ex1} alternate locally between $0$ and $1$ (see Theorem 2 of \cite{CK00}). \\
		Interestingly, the well known almost sure convergence to $\mathbf 0$ for
the Anderson model from Example \ref{pam} (see for instance \cite{GdH07}) does not contradict the findings here: the weak limit law is $\delta_{\mathbf{0}}$ so that the only point of the 		support is $\mathbf{0}$. Hence, the techniques of \cite{CK00} only  show that solutions locally approach $0$ infinitely often which is much weaker than the well known exponentially fast almost sure convergence to 0.

 \subsubsection{Longtime Analysis for the Moments - The Critical Moment Curve}
		
	Here we follow the arguments of \cite{BDE11} to find bounds for $p$th moments of $\textrm{SBM}_\gamma$ started at homogeneous initial conditions and for $p$th moments of the total-masses for 	compactly 		supported initial conditions. The calculations appearing here are the main building blocks for $\textrm{SBM}_\infty$ so that this section is more elaborate.
	\smallskip
	
	We start with a simple self-duality based lemma connecting the moments of solutions with infinite initial conditions with the moments of the total-mass processes for solutions starting at finite initial conditions.
	
	\begin{lem}\label{lem:11}
		For any $k\in \mathbb{Z}^d$ and $t\geq 0$
		\begin{align}\label{23}
			\E^{\1,\1} \big[ (u_t(k)+v_t(k))^p \big] = \E^{{\bf 1}_k,{\bf 1}_k} \big[ (\langle \1, {u}_t \rangle + \langle \1, {v}_t \rangle)^p \big].
		\end{align}
	\end{lem}
	\begin{proof}
		Employing the self-duality with $\phi=\psi=\frac{\theta}{2}{\bf 1}_{k}$ gives the Laplace transform identity
		\begin{align*}
			\E^{\1,\1} \big[ e^{-\sqrt{1-\varrho} \theta(u_t(k)+v_t(k))}\big]
			&=\E^{\1,\1}\big[e^{-\sqrt{1-\varrho}\langle u_{t}+v_{t},\phi+\psi\rangle}\big]\\
			&=\E^{\frac{\theta}{2}\1_{k},\frac{\theta}{2}\1_{k}}\big[e^{-\sqrt{1-\varrho}\langle \1+\1,{u}_{t}+{v}_{t}\rangle}\big]\\
			&=\E^{\1_{k},\1_{k}}\big[e^{-\sqrt{1-\varrho}\theta \langle \1, {u}_{t}+{v}_{t}\rangle}\big],
		\end{align*}
		where the final equality comes from the uniqueness of the solutions. 
	\end{proof}
	
	With identity (\ref{23}) in hand, to understand the behavior of the moments of the solutions starting at homogeneous initial conditions, it suffices to understand the behavior of the moments for the total-masses for which we already found a useful structure in Lemma 	\ref{la:eds}.
	
	A crucial ingredient is an exit-point exit-time equivalence for $\varrho$-correlated Brownian motions $(W^1,W^2)$ and their exit time $\tau$ from the first quadrant $\mathbb{Q}$ as defined in (\ref{exitd}), (\ref{exitt}).
	\begin{lem}[Exit-Point Exit-Time Equivalence]\label{ete}
		Let $p > 0$ and $u, v > 0$, then the following conditions are equivalent:
		\begin{itemize}
			\item[i)] $$p<p(\varrho):=\frac{\pi}{\frac{\pi}{2}+\arctan\big(\frac{\varrho}{\sqrt{1-\varrho^2}}\big)},$$ 
			\item[ii)] $$E^{u,v}\big[\tau^{\frac{p}{2}}\big]<\infty,$$
	 		\item[iii)] $$E^{u,v}\big[\big|(W^1_{\tau},W^2_{\tau})\big|^p\big]<\infty.$$
		\end{itemize}
	\end{lem}	
	The function $p$ is plotted in Figure 1; the critical curve $p$ is strictly decreasing, with $p(-1)=\infty$, $p(0)=2$ and $p(1)=1$.	
	 The particular case of $\varrho=0$ corresponds to planar Brownian motion in $\mathbb{Q}$. 
	  It is a classical result of Feller (see \cite{Fe}) that, independently of the initial 	value, the first hitting-time of the boundary only has $1-\epsilon$ finite moments. This part of the theorem is non-trivial! The second part is simpler as the exit point distribution can be calculated explicitly: with
	\begin{align*}
		\mathbb{Q}\rightarrow \mathbb{H},\,z\mapsto z^{2}
	\end{align*}
	the first quadrant is mapped conformally to the upper half plane $\mathbb{H}\subset \R^2$ so that, by conformal invariance, the planar Brownian path in $\mathbb{Q}$ is mapped to a time-changed planar Brownian 	path in 
	$\mathbb{H}$. 	Luckily, the time-change does not influence the exit-points (only the exit-time) and the exit-distribution from $\mathbb{H}$ is known to be Cauchy. Plugging this into the conformal mapping, the density of the exit-law $Q^0$ from 
	$\mathbb{Q}$ can be calculated (see page 1094 of \cite{DP98}). The density has no pole at zero with tail decreasing polynomially so that the number of finite moments can be deduced.
	
	\smallskip
	For $\varrho$-correlated Brownian motions the result follows from a simple change of the space; via 
	\begin{align*}
		(\tilde W^1,\tilde W^2):=\bigg(W^1,\frac{W^2-\varrho W^1}{\sqrt{1-\varrho^2}}\bigg),
	\end{align*}
	the $\varrho$-correlated Brownian motions are transformed into independent Brownian motions. Simultaneously, the quadrant $\mathbb{Q}$ is transformed into a wedge $\mathbb{W}(\theta)$ of angle 
	\begin{align*}
		\theta:=\frac{\pi}{2}+\arctan\Big(\frac{\varrho}{\sqrt{1-\varrho^2}}\Big)
	\end{align*}
	using the conformal map
	\begin{align*}
		\mathbb{Q}\rightarrow \mathbb{W}(\theta),\,z\mapsto z^{\pi/\theta}.
	\end{align*}	
	The angle of $\mathbb{W}(\theta)$ increases for increasing $\varrho$ explaining, at least morally, the decrease of the number of finite moments: if the domain is enlarged, the duration of a planar Brownian path to hit the boundary increases, hence, the exit-time might have less finite moments. At the same time, if 	the Brownian paths run for longer time it will hit larger values so that the hitting-point distribution might have less finite moments. Making this rigorous and calculating the exact number of finite moments is done in the same manner as for $\varrho=0$. That is, the exact distribution of the exit-law $Q^\varrho$ can be found:
	\begin{align}\label{h1}\begin{split}
		 		P^{u,v} \big( W^1_{\tau}=0,W^2_{\tau}\in dr) 
          			&=\frac{1}{\pi \sqrt{1-\varrho^2}^{\pi/\theta}}\frac{\frac{\pi}{\theta}r^{\frac{\pi}{\theta}-1}{z}_2
					}{{z}_2^2+\Big({ \big(\frac{r}{\sqrt{1-\varrho^2}}\big)^{\frac{\pi}{\theta}}
					+ {z}_1}{} \Big)^2 } \, dr,  \\ 
				P^{u,v} \big( W^1_{\tau}\in dr ,W^2_{\tau}=0 \big)
          			&=\frac{1}{\pi \sqrt{1-\varrho^2}^{\pi/\theta}} \frac{ \frac{\pi}{\theta} r^{\frac{\pi}{\theta}-1}{z}_2
				}{{z}_2^2+ \Big( {\big(\frac{r}{\sqrt{1-\varrho^2}}\big)^{\frac{\pi}{\theta}}-
				{z}_1} \Big)^2} \,dr, \end{split}
			\end{align}
	with the constants 		
			\begin{align}\label{h3}\begin{split}
			{z}_1 &= \Big( u^2\!+\frac{(v-\varrho u)^2}{1-\varrho^2} \Big)^{\frac{\pi}{2\theta}}
				\!\cos \Big( \frac{\pi}{\theta} \Big(\!\arctan \!
				\Big({\frac{v-\varrho u}{\sqrt{1-\varrho^2}u}} \Big)\! +
				\!\arctan \!\Big( \frac{\varrho}{\sqrt{1-\varrho^2}} \Big)\Big),\\
			{z}_2 &= \Big( u^2\!+\frac{(v-\varrho u)^2}{1-\varrho^2} \Big)^{\frac{\pi}{2\theta}}
				\!\sin \Big( \frac{\pi}{\theta} \Big( \!\arctan \!
				\Big( {\frac{v-\varrho u}{\sqrt{1-\varrho^2}u}} \Big) \!+
				\!\arctan \!\Big( \frac{\varrho}{\sqrt{1-\varrho^2}} \Big)\Big).\end{split}
			\end{align}			
	From the polynomial decay of the densities given in (\ref{h1}) the number of finite moments can be deduced.
	\begin{remark}
		The explicit density (\ref{h1}) will play a crucial part in Section 3. For the purposes of this section, the density serves as a tool to understand the longtime behavior of moments, whereas for $\textrm{SBM}_\infty$ it will be the main building block of a 		construction of the process.
	\end{remark}			
	
	A direct application of the exit-point exit-time equivalence is a proof for the critical moment curve of symbiotic branching processes. 	
	\begin{thm}[Critical Moment Curve]\label{thm:curve}
		Suppose  $\varrho\in(-1,1)$ and $\gamma>0$, then the following hold for $p>1$:
	\begin{itemize}
	 \item[a)] If $d=1,2$, then
		\begin{eqnarray*}
			p<p(\varrho) \quad \Longleftrightarrow\quad \E^{\1,\1}[u_t(k)^p] \text{ is bounded in }t\geq 0.
		\end{eqnarray*}
	\item[b)] If $d\geq 3$, then
		\begin{eqnarray*}
			p<p(\varrho)\,\quad \Longrightarrow \quad  \E^{\1,\1}\big[u_t(k)^p\big] \text{ is bounded in }t\geq 0.
		\end{eqnarray*}
	\end{itemize}
	By symmetry, the same statement holds for $u$ replaced by $v$. The inverse direction of b) fails and depends on $\gamma$.
	\end{thm}
	\begin{proof}[Sketch of Proof]
		Taking into account Lemma \ref{lem:11}, it suffices to prove the equivalence for the total-mass process $\langle u_t,\1\rangle$ started at localized initial condition. 
				
		\smallskip
		``$\Rightarrow$'': The proof basically follows from Lemma \ref{la:eds} and works for a) and b): The total-masses are time-changed $\varrho$-correlated Brownian motions and furthermore the quadratic variation (which is nothing but the time-change of the Brownian motions) is bounded by $\tau$ as otherwise one of the total-mass processes would become negative. Hence, by the 	Burkholder-Davis-Gundy inequality
		\begin{align*}
			 \E^{{\bf 1}_k,{\bf 1}_k} \big[ (\langle \1, {u}_t \rangle + \langle \1, {v}_t \rangle)^p \big]\leq CE^{1, 1}\big[\tau^{p/2}\big].
		\end{align*}
		The righthand side is independent of $t$ and finite due to the exit-point exit-time equivalence so that the claim follows.

	''$\Leftarrow$'': Suppose $p>p(\varrho)$. As in the proof of Theorem \ref{prop:convlaw} we use the almost sure convergence of $(\langle \1,{u}_t \rangle,\langle \1, {v}_t \rangle)$ to 
	$(W^1_{\tau}, W^2_{\tau})$ (this only works in case a), see Remark~\ref{rem:2807_1}). 
 Combining this with Fatou's Lemma gives
	\begin{align*}
		\liminf_{t\rightarrow \infty} \, \E^{{\bf 1}_k,{\bf 1}_k}\big[(\langle \1, {u}_t \rangle + \langle \1, {v}_t \rangle)^p\big]
		\geq \liminf_{t\rightarrow \infty}\E^{{\bf 1}_k,{\bf 1}_k}\big[\langle \1, {u}_t \rangle^p\big]
		\geq \E^{\1_k,\1_k}\big[\liminf_{t\rightarrow \infty}\langle \1,
		{u}_t \rangle^p\big] 
		=E^{1,1}\big[(W^1_{\tau})^p\big].
	\end{align*}
	The righthand side is infinite due to the exit-point exit-time equivalence so that the moment diverges. The results for $ \E^{\1,\1}\big[u_t(k)^p\big]$ can now be readily deduced by considering the cases $u_t(k)\leq v_t(k)$ and $v_t(k)\leq u_t(k)$.	
	\end{proof}

	\begin{figure}\label{fig}
	\begin{center}
        \hspace{-0.8cm} 
        \includegraphics[scale=0.25]{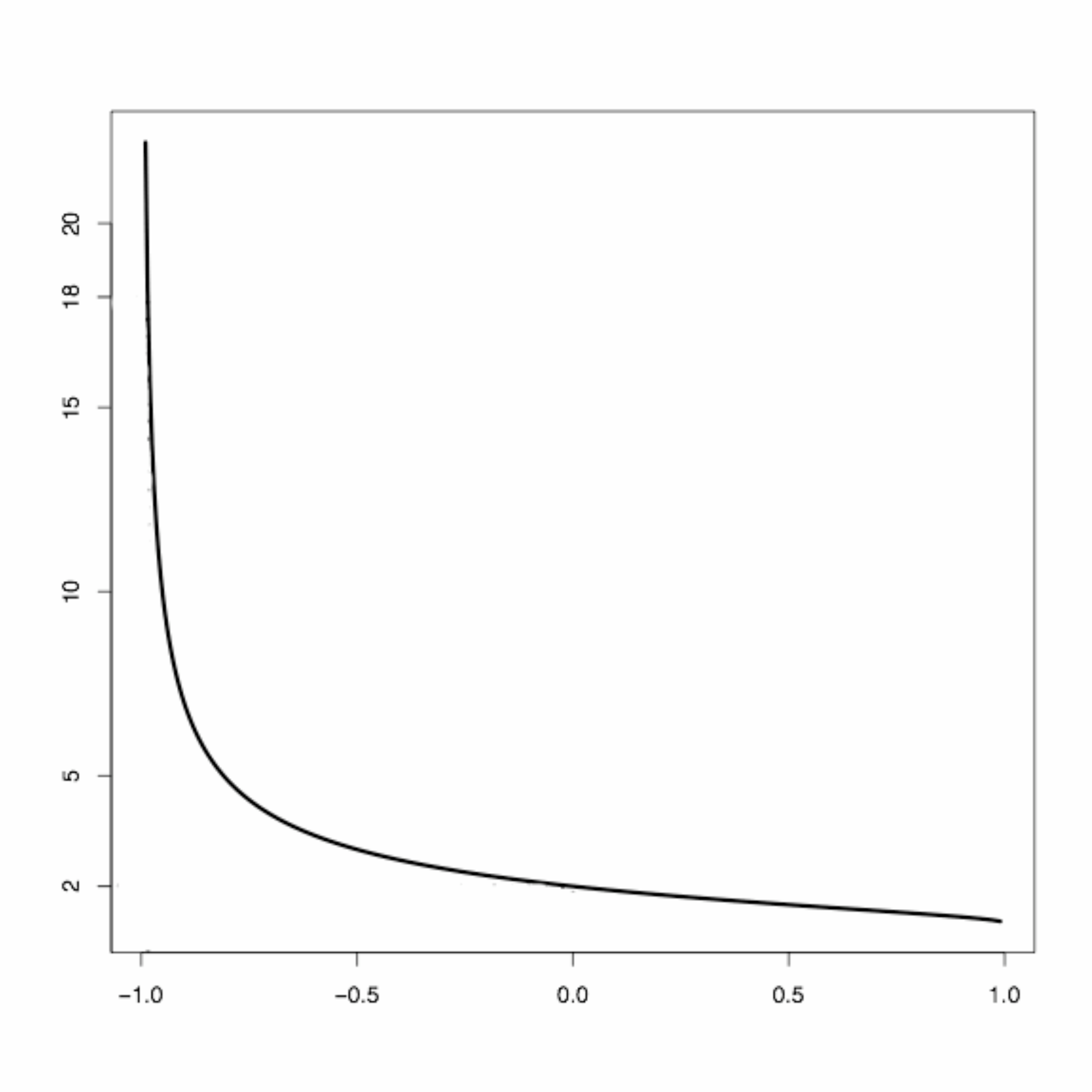}
      \end{center}
        \vspace{-.8cm} 
	\caption{The critical moment curve $p$ as a function of $\varrho \in (-1,1).$}
	\label{fig:cl}
\end{figure}



\subsection{Continuuum Model and the Interface Problem}\label{sec:wave}
	All results discussed above can equally be shown for the continuum space analogue model
 in low dimensions. We will briefly discuss this setting as it serves as an important motivation for the study of $\textrm{SBM}_\infty$.\\
	Let us first introduce the model for $d=1$. The continuum space symbiotic branching model is defined by the pair of stochastic heat equations
	\begin{equation}\label{eqn:spde}
		\begin{cases}
			\frac{\partial }{\partial t}u_t(x) = \frac{1}{2}\Delta u_t(x)+\sqrt{ \gamma u_t(x) v_t(x)} \, W^1(dt,dx),\\
			\frac{\partial }{\partial t}v_t(x) = \frac{1}{2}\Delta v_t(x)+\sqrt{ \gamma u_t(x) v_t(x)} \, W^2(dt,dx),
		\end{cases}
	\end{equation}
	where now $\Delta=\frac{\partial^2 }{\partial x^2}$ denotes the typical Laplace operator on $\R$. The driving noises ${W}^1,{W}^2$ are standard Gaussian white noises on $\mathbb{R}_+ \times \mathbb{R}$ with correlation parameter 
	$\varrho\in [-1,1]$, i.e. the unique Gaussian processes with covariance structure
	\begin{align*}
		\mathbb{E} \big[ W^1_{t_1}(A_1)W^1_{t_2}(A_2) \big] &=  (t_1\wedge t_2) \lambda (A_1\cap A_2),\\
		\mathbb{E} \big[ W^2_{t_1}(A_1)W^2_{t_2}(A_2) \big] &=  (t_1\wedge t_2) \lambda (A_1\cap A_2),\\
		\mathbb{E} \big[ W^1_{t_1}(A_1)W^2_{t_2}(A_2) \big] &= \varrho (t_1\wedge t_2) \lambda (A_1\cap A_2),
	\end{align*}
	where $\lambda$ denotes Lebesgue measure, $A_1,A_2\in\mathcal{B}(\R)$ and $t_1,t_2 \geq 0$.  Solutions of this model have been considered 	rigorously in the framework of the corresponding martingale problem in Theorem~4 of \cite{EF04}, which states that, under suitable growth conditions on the initial conditions $u_0, v_0$, a solution exists for all 
	$\varrho \in [-1, 1]$. Uniqueness for $\varrho\in (-1,1)$ can be obtained via the self-duality as in the proof of Corollary \ref{uniq}. The moment duality also holds with particles moving as Brownian motions and collision times replaced 	by collision local times.
	\smallskip
	
	Stochastic heat equations typically have function-valued  solutions only in spatial dimension $d=1$. The particular symmetric nature of $\textrm{MCB}_\gamma$ changes this property: it was shown in \cite{DEFMPX1} and \cite{DEFMPX2} that  
	 $\textrm{MCB}_\gamma$  do exist in the  continuous setting in dimension $d=2$ for $\gamma$ small enough. Existence 
 of solutions in dimensions $d>2$ is unknown. 
	\smallskip
	
	The results on the longtime behavior will not be repeated here; those are similar to the results discussed for the discrete spatial case for $d=1,2$. Instead, we include a result of \cite{BDE11} refining a Theorem of \cite{EF04}.
	 To explain this,  the notion of the interface of continuous-space symbiotic branching processes is needed.
	\begin{definition}\label{def:ifc}
		The interface at time $t$ of a solution $(u_t,v_t)$ of the symbiotic branching model $\mathrm{SBM}_\gamma$ with $\varrho \in [-1,1]$ is defined as 
		\begin{eqnarray*}
			\mbox{Ifc}_t = \mbox{cl} \big\{x : u_t(x) v_t(x) > 0 \big\},
		\end{eqnarray*}
		where $\mbox{cl}(A)$ denotes the closure of the set $A$ in $\R$.
	\end{definition}

	The main question addressed in \cite{EF04} is whether for complementary Heaviside initial conditions 
		\[
 	u_0(x) = {\bf 1}_{\R^-}(x) \quad \mbox{ and } \quad v_0(x) = {\bf 1}_{\R^+}(x).
	\]	
	the so-called compact interface property holds, that is, whether the interface is compact at each time almost surely. This is answered 		affirmatively in Theorem~6 in \cite{EF04}, together with the assertion that the interface propagates with at most linear speed, i.e. for each $\varrho \in [-1,1]$ there exists a constant $c>0$ and a finite random-time $T_0$ so that 	almost surely for all $T \ge T_0$
	\begin{align}\label{EF}
 	\bigcup_{t \le T} \mbox{Ifc}_t \subseteq \big[-cT, cT\big].
	\end{align}
	For the stochastic heat equation with Wright-Fisher noise corresponding to $\varrho=-1$, it was shown in \cite{T95} that the correct propagation of the interface is of order $\sqrt{T}$ so that one might ask whether (\ref{EF}) is sharp for $\varrho>-1$. Here is a refinement of (\ref{EF}), proved in \cite{BDE11}, for which the critical moment curve was originally developed.
	
	\begin{thm}\label{cor:wavespeed}
		Suppose $\varrho$ is chosen sufficiently small such that $p(\varrho)>35$ and $\gamma>0$, then there is a constant $C>0$ and a finite random-time $T_0$ such that almost surely for all $T>T_0$. 
		\begin{eqnarray*}
			\bigcup_{t\leq T} \mathrm{Ifc}_t \subseteq \Big[-C\sqrt{T\log(T)},C\sqrt{T\log(T)} \Big].
		\end{eqnarray*}
		
	\end{thm}
	The strong restriction on $\varrho$ is probably not necessary and is only caused by the technique of the proof which is based on the dyadic grid technique utilized for the proof of \cite{T95}. To circumvent the boundedness of all moments that holds only for $\varrho=-1$, $35$th moments have to be bounded in time.\\
	 Though the assumption forces $\varrho\leq  -0.9958$ the result is still interesting. It shows that sub-linear 		speed of propagation is not restricted to situations in which solutions are uniformly bounded as they are for $\varrho=-1$. 
	\bigskip
		
	Finally, let us motivate the construction and the study of $\textrm{SBM}_\infty$ in Section 3. The scaling property for symbiotic branching on the continuum (see Lemma 8 of \cite{EF04}) states that if $(u_t,v_t)$ is a solution started at Heavyside 	initial conditions, then
	\begin{align*}
		(\tilde u_t(x),\tilde v_t(x)):=\big(u_{Nt}\big(\sqrt{N}x\big),v_{Nt}\big(\sqrt{N}x\big)\big)
	\end{align*}
	is a solution of $\textrm{SBM}_{\sqrt{N}\gamma}$ with Heavyside initial condition. Hence, propagation of the interface of order $\sqrt{T}$ will be intimately related to the behavior of $\textrm{SBM}_{\gamma}$ with $\gamma$ 		tending to infinity.\\
	Unfortunately, the constructions in Section 3 can only be seen as a first step towards the correct order of interface propagation: the construction for the limiting process $\textrm{SBM}_\infty$ could so far be carried out only for discrete spatial symbiotic 	branching processes. It is still an open question how to extend the characterizations and constructions of $\textrm{SBM}_\infty$ to the continuum analogue.

\newpage
\section{Infinite Rate Symbiotic Branching Processes}\label{sec:3}
 In Section~\ref{sec:voter1} we discussed how the standard voter processes can be viewed as an infinite rate stepping stone model, or, in other words, $\textrm{SBM}_\infty(\varrho)$ for $\varrho=-1$. 
	It is not at all clear if and how that motivation  extends to $\varrho\neq -1$ as the coalescing particles duality seems to have no extension to $\varrho\neq -1$. Taking into account the colored 		particles dual instead, it is by no means clear whether sending $\gamma$ to infinity leads to a non-trivial process: for $\gamma=\infty$ the changes of color occur instantaneously but at the same time the exponent is multiplied 	by $\infty$, so that the moment expression only makes sense if the exponent is almost surely non-positive.\\
	Nonetheless, using the self-duality instead of the moment-duality, it can be shown that sending the branching rate to infinity makes sense. To understand the effect in a nutshell, let us take a closer look at the non-spatial system of 	symbiotic	branching SDEs
	\begin{align}\label{sde}\begin{cases}
		du_t=\sqrt{\gamma u_tv_t}\,dB_t^1,\\
		dv_t=\sqrt{\gamma u_tv_t}\,dB_t^2,\end{cases}
	\end{align}
	with non-negative initial conditions $(u,v)$. Due to the symmetric structure, we got in Lemma \ref{la:eds} that
	\begin{align*}
		\big( W_{t}^1,W_{t}^2\big):=\big(u_{T^{-1}(t)},v_{T^{-1}(t)}\big)
	\end{align*}	
	are $\varrho$-correlated Brownian motions if we use the time-change $T(t)=\gamma \int_0^tu_sv_s\,ds$. Caused by the product structure of the time-change the boundary $E$ of the first quadrant is absorbing. Hence, the Brownian motions $W^1,W^2$ stop at the first hitting-time $\tau$ of $E$. Increasing $\gamma$ only has the effect that $(u_t,v_t)$ follows the Brownian paths with different speed so that $\gamma=\infty$ corresponds to at once picking a point in $E$ according to the exit-measure $Q^\varrho$ on $E$ and freeze thereafter (recall (\ref{exitd})).\\ To make this argument precise one has to be slightly more careful as the parameter $\gamma$ does not only occur as multiple in the time-change $T$ but also effects the solution itself. To circumvent this obstacle one has to take into account the structure of the equations. Let us label the solutions by their fixed branching rate $\gamma$.  It can be shown that the sequence $(u^\gamma,v^\gamma)$ converges in the so-called Me
 yer-Zheng ``pseudo-path" topology (for which we refer to \cite{MZ} and \cite{J97})   to a limit $(U,V)$. Stochastic boundedness in $\gamma$ and $t$  of the square-function
		$\gamma \int_0^t u_s^\gamma v_s^\gamma ds$
		by $\tau$ implies that 
		$\int_0^t U_sV_s\,ds=0.$
	 Hence, the limiting process $(U,V)$ takes values in $E$. The only possible limit is the constant process $(U_t,V_t)=(U_0,V_0), t\geq 0$, where $(U_0,V_0)$ is distributed according to $Q^\varrho_{u,v}$
	 because the prelimiting processes $(u^\gamma,v^\gamma)$ are eventually trapped at $E$ at a point distributed according to $Q^\varrho_{u,v}$.
	 	\medskip
	
	Incorporating space, a second effect occurs: both types change their mass on $S$ according to a heatflow. This smoothing effect immediately tries to lift a zero coordinate if it was pushed by the exit-measure $Q^\varrho$ to zero. 
	Interestingly, none of the two effects dominates and a non-trivial limiting process (with values in $E$ for each site $k\in S$) can be obtained when letting the branching rate tend to infinity.	
	\begin{convention}
		In contrast to Section \ref{sec:2} we do not restrict to the discrete Laplacian $\Delta$ here and instead replace $\Delta$ by $\mathcal A$ as in Section \ref{sec:uni}. Accordingly, $\Z^d$ is replaced by a general countable set 		$S$.
	\end{convention}
	The aim of this section is to explain how the results of  \cite{KM10b} and \cite{KO10} on the infinite rate mutually catalytic branching process $\textrm{MCB}_\infty$ can be generalized to $\varrho\neq 0$.
	 After introducing more notation for the state-spaces, different approaches to infinite rate symbiotic branching processes are presented: a characterization via an abstract martingale problem, two limiting constructions and a more hands-on representation via Poissonian integral equations. 
\subsubsection{Some Notation}
	The finite rate symbiotic branching processes were studied on subspaces of $\big(\R^+\times \R^+\big)^S$, i.e. at each site of the countable set $S$ the solution processes consist of a pair of non-negative values. According to the heuristic reasoning above, at each site $k\in S$ infinite rate processes take values on the boundary $E$ of the first quadrant so that we can expect to find an $E^S$-valued process. As usual, certain growth restrictions need to be imposed to find a tractable subspace of $E^S$. In accordance with the state-space $L_\beta^2$ for finite rate 	symbiotic branching processes we stick to the analogue subspace of $E^S$:
	\begin{align}\label{111}
		L_\beta^{2,E}:=L_\beta^2\cap E^S
	\end{align}
	equipped with the same norm as $L^2_\beta$.
	Furthermore, we will use subspaces of compactly supported and summable initial conditions that will be denoted by $L^{f,E}$ and $L^{\Sigma,E}$. In contrast to $\textrm{SBM}_\gamma$, the infinite rate processes are not continuous so that solutions have paths in $D\big([0,\infty),L^{2,E}_\beta\big)$, the set of functions that are right-continuous with limits from the left.

\subsection{Martingale Properties}\label{sec:mp}

	In order to define infinite rate processes rigorously, in \cite{KM10b} a martingale problem characterization was proposed for infinite rate mutually catalytic branching processes. This formulation uniquely determines the process but is not very useful for understanding properties of the process. Crucial properties of the process, such as non-continuity of sample paths, are not clear from this formulation. Nonetheless, it seems to be the most convenient way to introduce the 		process as it directly reveals the connection to the finite rate processes. In what\ follows we are going to extend the results of \cite{KM10b} to $\varrho\neq 0$. 
		
	\smallskip
	To define the characterizing martingale problem one crucially uses the self-duality function 
	\begin{align}\label{F}
		F(x_1,x_2,y_1,y_2)&=\exp\big(\langle \langle x_1,x_2,y_1,y_2 \rangle\rangle_\varrho\big)
	\end{align}
	defined in (\ref{sd}). We include the next two simple (stochastic) calculus lemmas in order to clarify the appearance of $\varrho$ in the definition of $\langle \langle \cdot,\cdot,\cdot,\cdot\rangle\rangle_\varrho$. 
		\begin{lem}\label{l6}
			Suppose $(x_1,x_2)\in L^2_{\beta}$ and $(y_1,y_2)$ are compactly supported, then for all $k\in S$
			\begin{align*}
				\frac{\partial}{\partial x_1(k)}F(x_1,x_2,y_1,y_2)&=F(x_1,x_2,y_1,y_2) \langle \langle 1,0,y_1,y_2 \rangle\rangle_\varrho\\
				\frac{\partial}{\partial x_2(k)}F(x_1,x_2,y_1,y_2)&=F(x_1,x_2,y_1,y_2) \langle \langle 0,1,y_1,y_2 \rangle\rangle_\varrho
			\end{align*}
			and
			\begin{align*}
				&\quad\left[\frac{1}{2}\frac{\partial^2}{\partial x_1(k)^2}+\frac{1}{2}\frac{\partial^2 }{\partial x_2(k)^2}+\varrho \frac{\partial^2}{\partial x_1(k)\partial x_2(k)} \right]F(x_1,x_2,y_1,y_2)
				=4(1-\varrho^2)F(x_1,x_2,y_1,y_2)y_1(k)y_2(k),
			\end{align*}
			 where $\partial/\partial x_1(k)$ (resp. $\partial/\partial x_2(k)$) denotes the partial derivative with respect to the $k$th coordinate of the first (resp. second) entry.
		\end{lem}
		\begin{proof}
			First note that all appearing infinite sums are actually finite as $y_1$ and $y_2$ are compactly supported. We leave the simple derivations of the first derivatives to the reader as it does not clarify the influence of 
			$\varrho$.\\
			Abbreviating $c(k)=y_1(k)+y_2(k)$ and $d(k)=y_1(k)-y_2(k)$, by the chain rule we obtain
			\begin{align*}
				&\quad\left[\frac{1}{2}\frac{\partial^2}{\partial x_1(k)^2}+\frac{1}{2}\frac{\partial^2 }{\partial x_2(k)^2}+\varrho \frac{\partial^2}{\partial x_1(k)\partial x_2(k)} \right]F(x_1,x_2,y_1,y_2)\\
				&=F(x_1,x_2,y_1,y_2)
				\Bigg[\frac{1}{2}\Big(-\sqrt{1-\varrho}c(k)+i\sqrt{1+\varrho}d(k)\Big)^2+\frac{1}{2}\Big(-\sqrt{1-\varrho}c(k)-i\sqrt{1+\varrho}d(k)\Big)^2\\
				&\quad+\varrho\Big(-\sqrt{1-\varrho}c(k)+i\sqrt{1+\varrho}d(k)\Big)\Big(-\sqrt{1-\varrho}c(k)-i\sqrt{1+\varrho}d(k)\Big)\Bigg]
			\end{align*}
			which is equal to
			\begin{align*}
				&\quad F(x_1,x_2,y_1,y_2)\big[(1-\varrho)(1+\varrho)c^2(k)-(1+\varrho)(1-\varrho)d^2(k)\big]\\
				&=F(x_1,x_2,y_1,y_2)4(1-\varrho^2)y_1(k)y_2(k).
			\end{align*}
		\end{proof}	
	The intrinsic need for the particular choice of $\langle\langle\cdot,\cdot, \cdot,\cdot\rangle\rangle_\varrho$ can now be revealed: the additional square-roots involving $\varrho$ are chosen in such a way that the cross-variations caused by the correlated driving noises cancel.
	\begin{prop}\label{prop:mart1}
		Suppose $(u_0,v_0)\in L^2_{\beta}$, $(y_1,y_2)\in L^{f,E}$ and $(u_t,v_t)_{t\geq 0}$ is a symbiotic branching process with finite branching rate $\gamma>0$ and correlation parameter $\varrho \in [-1,1]$, then
		\begin{align}\label{mp1}\begin{split}
			&\quad M^{\varrho,\gamma}_t(u_0,v_0,y_1,y_2)\\
			&:=F(u_t,v_t, y_1,y_2)-F(u_0,v_0, y_1,y_2)-\int_0^t\langle \langle \mathcal A u_s,\mathcal A v_s,y_1,y_2\rangle\rangle_\varrho F(u_s,v_s, 				y_1,y_2)\,ds,\quad t\geq 0,
\end{split}		\end{align}
		is a martingale null at zero.
	\end{prop}
	\begin{proof}
		Noting again that all infinite sums are, in fact, finite as the test-sequences $y_1$ and $y_2$ have compact support, we may apply It\^o's formula to the finite set of stochastic differential equations to get
		\begin{align*}
			&\quad F(u_t,v_t,y_1,y_2)\\
			&=F(u_0,v_0,y_1,y_2)+\sum_{k\in S}\int_0^t\frac{\partial}{\partial x_1(k)}F(u_s,v_s,y_1,y_2)\,du_s(k)\\
			&\quad+\sum_{k\in S}\int_0^t\frac{\partial}{\partial x_2(k)}F(u_s,v_s,y_1,y_2)\,dv_s(k)
			+\frac{1}{2} \sum_{k\in S}\int_0^t\frac{\partial^2}{\partial x_1(k)^2}F(u_s,v_s,y_1,y_2)d\langle u_\cdot(k)\rangle_s\\
			&\quad+\frac{1}{2}\sum_{k\in S}\int_0^t \frac{\partial^2}{\partial x_2(k)^2}F(u_s,v_s,y_1,y_2)d\langle v_\cdot(k)\rangle_s
			+ \sum_{k\in S}\int_0^t\frac{\partial^2}{\partial x_1(k)\partial x_2(k)}F(u_s,v_s,y_1,y_2)d\langle u_\cdot(k),v_\cdot(k)\rangle_s,
		\end{align*}
		where we used that by definition the Brownian motions at different sites are independent. 
The correlation structure for the Brownian motions at the same sites and the previous lemma yield equality of the above expression to
		\begin{align*}
			& F(u_0,v_0,y_1,y_2)+\text{local mart.}
			+\int_0^tF(u_s,v_s,y_1,y_2) \langle \langle \mathcal A u_s,0,y_1,y_2 \rangle\rangle_\varrho\,ds
			+\text{local mart.}\\
			&+\int_0^tF(u_s,v_s,y_1,y_2) \langle \langle 0, \mathcal A v_s,y_1,y_2 \rangle\rangle_\varrho\,ds
			+\int_0^tF(u_s,v_s,y_1,y_2)\sum_{k\in S}4(1-\varrho^2)y_1(k)y_2(k)\gamma u_s(k)v_s(k)\,ds.
		\end{align*}
		Sorting the terms leads to
		\begin{align*}
			 F(u_t,v_t,y_1,y_2)&=F(u_0,v_0,y_1,y_2)+\text{local mart}.+\int_0^t F(u_s,v_s,y_1,y_2)\langle \langle \mathcal A u_s, \mathcal A v_s, y_1,y_2\rangle \rangle_\varrho\,ds\\
			&\quad+\int_0^tF(u_s,v_s,y_1,y_2)\sum_{k\in S}4(1-\varrho^2)y_1(k)y_2(k)\gamma u_s(k)v_s(k)\,ds.
		\end{align*}
		By assumption $y\in L^{f,E}$ so that $y_1(k)y_2(k)=0$ for all $k\in S$. Hence, the last summand vanishes and it only remains to show that the local martingale is a martingale. But this follows directly from the fact that $F$ is 		bounded and the moments $\E[u_t(k)v_t(k)]$ are locally bounded. The latter follows for instance from the moment duality of Lemma \ref{la:mdual}.
	\end{proof}
	It would be desirable to uniquely define solutions of finite rate symbiotic branching processes via this martingale property which unfortunately is impossible: the corresponding martingale problem does not involve $\gamma$ and it is satisfied by $\textrm{SBM}_\gamma(\varrho)$ for arbitrary $\gamma$. As symbiotic branching processes for different branching rates do not coincide in law, the martingale problem has infinitely many solutions.\\
	However, the class of processes on the restricted state-space $L_\beta^{2,E}$ is less rich so that the small class of test-functions suffices here for the martingale problem to be well-posed. In particular, the restriction rules out all solutions of $\textrm{SBM}_\gamma$. Here is the generalization from $\varrho=0$ to $\varrho\in (-1,1)$ of Proposition 4.1 of \cite{KM10b}.
	
	\begin{thm}\label{pro:1}
		Let $\varrho\in (-1,1)$, then there is a unique solution to the following martingale problem: For all initial conditions $(x_1,x_2)\in L_\beta^{2,E}$, there exists a process $(U,V)$ with paths in  
		$D([0,\infty),L_\beta^{2,E})$ such that for all 
		test-sequences $(y_1,y_2)\in L^{f,E}$ the process
		\begin{align}\label{mp}
			M^{\varrho,\infty}_t(x,y)
			&:=F(U_t,V_t, y_1,y_2)-F(x_1,x_2, y_1,y_2)-\int_0^t\langle \langle \mathcal A U_s,\mathcal A V_s,y_1,y_2\rangle\rangle_\varrho F(U_s,V_s,y_1,y_2)\,ds 
		\end{align}
		is a martingale null at zero. The induced law on $D\big([0,\infty),L_\beta^{2,E}\big)$ constitutes a strong Markov family and the corresponding strong Markov process will be called infinite rate symbiotic branching {$\textrm{SBM}_\infty(\varrho)$}.
	\end{thm}
	We postpone a sketch of a proof to Section \ref{sec:jumpsde} where solutions are constructed by means of the Poissonian equations already mentioned in Theorem \ref{0}.

\smallskip

Since we discussed extensively the longtime behavior of finite rate symbiotic branching processes we say a few words about the longtime behavior of infinite rate symbiotic branching processes. The case 
of $\varrho=0$ has been studied in~\cite{KM10a} and some sufficient conditions for coexistence and impossibility 
of coexistence have been derived there. For $\varrho<0$ a full recurrence/transience dichotomy has been established in~\cite{DM11} in the spirit of the results presented in Section \ref{longl}.
\begin{prop} 
\label{prop:110}
		Let $\varrho<0$, then coexistence of types for $\textrm{SBM}_\infty(\varrho)$ is possible if and 
only if a Markov process 
 on $S$ with $Q$-matrix ${\mathcal A}$ is transient. 
\end{prop}
Note that this proposition extends Proposition~\ref{prop:101} to $\gamma=\infty$ on a general countable site space $S$ and an arbitrary 
 symmetric Markov process with $Q$-matrix ${\mathcal A}$. For the proof we refer the reader  to~\cite{DM11}.
	
\subsection{Main Limit Theorem}
	So far we have discussed the finite rate symbiotic branching processes and introduced the well-posed martingale problem from which one can define the family of processes $\textrm{SBM}_\infty(\varrho)$, $\varrho\in (-1,1)$. To get the link between the two, we sketch in this section how to show that $\textrm{SBM}_\gamma$ converges in some weak sense to the solution $\textrm{SBM}_\infty$ of the martingale problem (\ref{mp}) as $\gamma$ goes to infinity. This, in fact, justifies to call the processes of Theorem \ref{pro:1} infinite rate symbiotic branching processes.
\smallskip

	Unfortunately, the convergence of $\textrm{SBM}_\gamma$ to $\textrm{SBM}_\infty$ will not hold in the convenient Skorohod topology in which continuous processes converge to continuous processes. As a solution of the system 	of Brownian equations (\ref{ss}), $\textrm{SBM}_\gamma$ is continuous, whereas $\textrm{SBM}_\infty$ is non-continuous as solution to the system of Poissonian equations.
	\smallskip
	
	Even though the convergence can not hold in the Skorohod topology, it holds in some weaker sense. The suitable ``pseudo-path" topology on the Skorohod space of RCLL functions was introduced in \cite{MZ}. 	The topology is much weaker than the Skorohod topology and is, in fact, equivalent to convergence in measure (see Lemma 1 of \cite{MZ} and also results in \cite{J97}). Sufficient 	(but not necessary) tightness conditions for this ``pseudo-path" topology were given in \cite{MZ}. In particular, these conditions are convenient to check the tightness of semimartingales.
	\smallskip
	
	Here is the extension of Theorem 1.5 of \cite{KM10b} to $\varrho\neq 0$.
\begin{thm}\label{thm:2}
	Fix any $\varrho\in (-1,1)$. Suppose that for any $\gamma>0$, $(u^{\gamma}_t,v^{\gamma}_t)_{t\geq 0}$ solves $\textrm{SBM}_{\gamma}(\varrho)$ and the initial conditions $(u^\gamma_0,v^\gamma_0)=(U_0,V_0)\in L^{2,E}_\beta$ do not depend on $\gamma$. Then, for any sequence $\gamma_n$ tending to infinity, we have the convergence in law
	\begin{align*}
		(u^{\gamma_n},v^{\gamma_n})\Longrightarrow (U,V),\quad n\to\infty,
	\end{align*}
	in $D([0,\infty),L_\beta^2)$ equipped with the Meyer-Zheng ``pseudo-path" topology. Here, $(U,V)$ is the unique solution of the martingale problem of Theorem \ref{pro:1}.
\end{thm}

\begin{proof}[Sketch of Proof]
	The proof consists of three steps:\\
	\textbf{Step 1:} Tightness in the Meyer-Zheng ``pseudo-path" topology follows from the tightness criteria of \cite{MZ}. To carry this out, one has to show tightness for the drift and the martingale terms  in the definition of $\textrm{SBM}_\gamma$: By standard estimates 	the drift terms are, in fact, tight in the stronger Skorohod topology: this follows from
	\begin{align}\label{amo}
		\sup_\gamma \E\Big[\sup_{t\leq T} \langle u_t^\gamma,\beta\rangle^p\Big]<\infty,\quad p\in (1,p(\varrho)).
	\end{align}
	Apart from the facts that $\beta\neq 1$ and $u_0^\gamma$ is not assumed to be summable this is close to the moment bounds for the total-mass processes that we obtain from Lemma \ref{lem:11} and Theorem \ref{thm:curve}. With the same trick as in Lemma 6.1 of \cite{KM10b}, the lefthand side of (\ref{amo}) can be bounded uniformly in $\gamma$ by a multiple of $\E[\tau^{p/2}]$, where $\tau$ is the exit-time of Theorem \ref{thm:curve}. Replacing $p\in (1,2)$ by $p\in (1,p(\varrho))$, the arguments in the proof of Lemma 6.2 of \cite{KM10b} carry over line by line. The crucial observation is that for all $\varrho\in (-1,1)$ the critical curve $p(\varrho)$ is strictly larger than $1$ which is all that is needed. \\
	To prove tightness of the martingale part, the tightness criteria for martingales can be applied (compare Theorem 4 combined with Remark 2 of \cite{MZ}).\\ 
	\textbf{Step 2:} To show that all limit points indeed solve the martingale problem, we only have to use Proposition \ref{prop:mart1} and some moment estimates. For all fixed $\gamma>0$ the same martingale problem is fulfilled so that it comes as no 	surprise that the martingale problem is fulfilled in the limit if one can show that the martingales converge to a martingale. But this follows from the same estimates that are used for the tightness proof involving crucially the critical curve $p(\varrho)$.\\
	\textbf{Step 3:} In the previous section we stressed out that the martingale problem (\ref{mp}) is only well-posed if the involved process takes values in the restricted space $L_\beta^{2,E}$. To show that for any limit point $(U,V)$, we indeed have $U_t(k)V_t(k)=0$ for all $t\geq 0,k\in S$, one can show that almost surely
	\begin{align}\label{lll}
		\int_0^t(U_s(k) V_s(k)\wedge 1)\,ds=0,\quad \forall t\geq 0,k\in S,
	\end{align}
	since then, by right-continuity, $(U_t(k),V_t(k))\in E$ for all $t\geq 0, k\in S$. By tightness of step 1 one can easily derive stochastic boundedness of $\gamma\int_0^t u^\gamma_s v^\gamma_s\,ds$ uniformly in $\gamma$ for any $k\in S$ as in the proof of Lemma 6.3 in \cite{KM10b} from which (\ref{lll}) follows by taking into account that convergence in the Meyer-Zheng ``pseudo-path" topology is equivalent to convergence in (Lebesgue) measure.	
	\end{proof}
Now that above we have made precise sense of $\textrm{SBM}_\infty$ in terms of a weak limit of $\textrm{SBM}_\gamma$ that solves a well-posed martingale problem, the next two sections are devoted to constructions that shed more light on the properties of the processes. 

\subsection{Trotter Type Construction}
	A very different perspective for $\textrm{MCB}_\infty$ was presented in \cite{KO10}. Their main idea was to combine ``by hands" the precise  infinite rate limit for the mutually catalytic SDE (\ref{sde}) with the heatflow corresponding to the generator $\mathcal A$, to construct a more instructive approximation. The approximation converges in the stronger Skorohod topology, instead of only in the weaker Meyer-Zheng ``pseudo-path" topology, which might be helpful to deduce properties for the limiting process. We now briefly discuss here how their approach extends to $\textrm{SBM}_\infty(\varrho)$ for $\varrho\in (-1,1)$.\\
	Separating the deterministic and stochastic terms in the very definition of $\textrm{SBM}_\gamma$, one has to consider the pair of evolution equations
		\begin{align}\label{jj}
			\frac{\partial }{\partial t}u_t=\mathcal A u_t,\qquad \frac{\partial }{\partial t}v_t=\mathcal Av_t
		\end{align}
		and the set of independent two-dimensional symbiotic branching processes
		\begin{align}\label{j}
			du_t(i)=\sqrt{\gamma u_t(i)v_t(i)}\,dB^1_t(i),\qquad dv_t(i)=\sqrt{\gamma u_t(i)v_t(i)}\,dB^2_t(i).
		\end{align}
		The evolution equations (\ref{jj}) can be solved explicitly in terms of the semigroup $P_t$ corresponding to $\mathcal A$ (recall (\ref{SG}) for $\mathcal A=\Delta$) and the solutions do not depend on the branching rate $\gamma$.
		The processes in (\ref{j}) obey a more interesting behavior as we have discussed in the introduction of this section: the pairs of independent stochastic integrals provide a set of independent diffusions indexed by $S$ which, for $\gamma=\infty$, correspond to a set of independent choices of the exit-law $Q^\varrho$.
		
		The Trotter type approach to $\textrm{SBM}_\infty$ is built upon these two explicit representations: the heatflow based on $\mathcal A$ and the independent choices of the exit-law $Q^\varrho$ are alternated with increasing frequency. For $\epsilon>0$, the approximating 			processes $(U^\epsilon,V^\epsilon)$ are defined as follows:
		\begin{itemize}
			\item[i)] Within each interval $[n\epsilon, (n+1)\epsilon)$, $(U_t^\epsilon,V_t^\epsilon)$ is the explicit (deterministic) solution of (\ref{jj}) with initial condition $(U^\epsilon_{n\epsilon},V^\epsilon_{n\epsilon})$.
			\item[ii)] At times $n\epsilon$, $(U^\epsilon_{n\epsilon-},V^\epsilon_{n\epsilon-})$ is replaced at each site $k\in S$ independently by a point in $E$ chosen from the exit-law $Q^\varrho$ of $\varrho$-correlated Brownian motions started at $\big(U^{\epsilon}_{n \epsilon -}(k),V^{\epsilon}_{n \epsilon -}(k)\big)$.
		\end{itemize}
		It follows almost directly from the definition of the approximation that for any $\epsilon>0$ fixed, $(U^\epsilon,V^\epsilon)$ is a solution to the martingale problem (\ref{mp}). As discussed below Proposition \ref{prop:mart1} this does not cause 		any contradiction since even with initial condition $(U_0,V_0)\in L_\beta^{2,E}$, the processes $(U^\epsilon,V^\epsilon)$ take values in $L_\beta^2$ but not in the restricted state-space $L_\beta^{2,E}$. The adaption of the main result of \cite{KO10} is then as follows.
		\begin{thm}\label{thm:trotter}
			Suppose $\varrho\in(-1,1)$ and $(U_0,V_0)\in L_\beta^{2,E}$. Then, as $\epsilon$ tends to zero, the family $\{(U^\epsilon,V^\epsilon)\}_{\epsilon>0}$  converges weakly, in the Skorohod topology on $D([0,\infty),L_\beta^2)$, to $\textrm{SBM}_\infty(\varrho)$.
		\end{thm}
		\begin{proof}[Sketch of Proof]
			The proof consists of three steps:\\
			\textbf{Step 1:} Tightness in the Skorohod topology is proved via Aldous' criterion and moment estimates that are based on estimates for the exit-measures $Q^\varrho$. The extension from $\varrho=0$ to $\varrho\in (-1,1)$ is crucially built upon the fact that the estimates of \cite{KO10} are based on boundedness of some moments greater than $1$ and this equally holds for any $\varrho\in(-1,1)$ (see Lemma \ref{ete}).\\
			\textbf{Step 2:} The identification of the limit points is not difficult since the approximating sequence already solves the martingale problem for any $\epsilon>0$. It is only needed to show that the sequence of 			martingales remains a martingale for which again moment estimates based on Lemma \ref{ete} are needed.\\
			\textbf{Step 3:} The limiting process takes values in the smaller space $L_\beta^{2,E}\subset L^2_\beta$ due to the construction and continuity of the heatflow.
		\end{proof}

\subsection{Poissonian Construction}\label{sec:jumpsde}
	Up to now the infinite rate symbiotic branching processes have only been characterized as weak limits of approximating sequences and via an abstract martingale problem. The most explicit construction
 of $\textrm{SBM}_\infty$ is presented here as 	the unique weak solution to a system of Poissonian integral equations from which we deduce the connection to the voter process.
\subsubsection{Jump Measure}
	To describe the jumps of $\textrm{SBM}_\infty$, the following definition is needed:	
	
	\begin{definition}\label{defj}
		Suppose $Q^\varrho_{(u,v)}$ is the exit-measure of $\varrho$-correlated Brownian motions started at $(u,v)$ from the first quadrant (see (\ref{exitd})). Define
		\begin{align*}
			\nu^{\varrho}_{(a,0)}&:=\lim_{\epsilon\to 0} \frac{Q^{\varrho}_{(a,\epsilon)}}{\epsilon},
		\end{align*}
		where the limit is in the vague topology on measures (i.e. integrated against continuous functions with compact support).
	\end{definition}	
	
	Recalling that by definition $Q^\varrho$ is a probability measure for arbitrary initial conditions, after rescaling the measure $\nu^\varrho$ has to be an infinite measure on $E$ equipped with the restricted Borel 
	$\sigma$-algebra. 
	
	\smallskip
	Next, a density for $\nu^\varrho$ will be derived. The core of the work has already been done in \cite{BDE11} where the density for $Q^\varrho$ was calculated (see (\ref{h1})); with this density 
in hand we get the following result. 

	\begin{lem}
		For $\varrho\in(-1,1)$ and $a>0$ the measure $\nu^\varrho$ is absolutely continuous with respect to the two-dimensional Lebesgue measure restricted to $E$ with the density
		\begin{align*}
			\nu_{(a,0)}^\varrho (d(y_1,y_2))=\begin{cases}
				{p(\varrho)^2} a^{p(\varrho)-1} 
\sqrt{1-\varrho^2}\frac{y_1^{p(\varrho)-1}}{\pi \big(y_1^{p(\varrho)}-a^{p(\varrho)}\big)^2}\,dy_1&: y_2=0,\\
				{p(\varrho)^2}a^{p(\varrho)-1}\sqrt{1-\varrho^2}\frac{y_2^{p(\varrho)-1}}{\pi\big(y_2^{p(\varrho)}+a^{p(\varrho)}\big)^2}\,dy_2&: y_1=0,\\
			\end{cases}
		\end{align*}
		where $p(\varrho)=\frac{\pi}{\theta(\varrho)}$ and $\theta(\varrho)=\frac{\pi}{2}+\arctan\Big(\frac{\varrho}{\sqrt{1-\varrho^2}}\Big)$.
	\end{lem}
	\begin{proof}
	By definition of $\nu^\varrho$, all we need to do is to plug-in $(u,v)=(a,\epsilon)$ into the the explicit density given in (\ref{h1}), divide by $\epsilon$ and go to the limit. With the notation used in~(\ref{h1}),  (\ref{h3}) we obtain
			\begin{align*}
			{z}_1(\epsilon) &= \Big( a^2+\frac{(\epsilon-a\varrho )^2}{1-\varrho^2} \Big)^{\frac{\pi}{2\theta(\varrho)}}\!\cos \Big( \frac{\pi}{\theta(\varrho)} \Big(\!\arctan \!\Big({\frac{\epsilon-a\varrho }{\sqrt{1-\varrho^2}a}} \Big)\! +
				\!\arctan \!\Big( \frac{\varrho}{\sqrt{1-\varrho^2}} \Big)\Big)\\
				&\stackrel{\epsilon\to 0}{\rightarrow}a^{p(\varrho)}\Big(\frac{1}{\sqrt{1-\varrho^2} }\Big)^{p(\varrho)}.
			\end{align*}
			Taking into account $\sin(x)\sim x$ at $0$ and l`H\^opital's rule with $\arctan'(x)=1/(1+x^2)$, leads to
			\begin{align*}
			\epsilon^{-1}{z}_2(\epsilon) &= \epsilon^{-1}\Big( a^2+\frac{(\epsilon-a\varrho )^2}{1-\varrho^2} \Big)^{\frac{\pi}{2\theta(\varrho)}}\sin \Big( \frac{\pi}{\theta(\varrho)} \Big( \!\arctan \!\Big( {\frac{\epsilon-a\varrho }{\sqrt{1-\varrho^2}a}} \Big) \!+
				\!\arctan \!\Big( \frac{\varrho}{\sqrt{1-\varrho^2}} \Big)\Big)\Big)\\
				&\stackrel{\epsilon\to 0}{\sim}a^{p(\varrho)}\Big( \frac{1}{1-\varrho^2} \Big)^{\frac{\pi}{2\theta(\varrho)}}\epsilon^{-1}
				\sin \Big( \frac{\pi}{\theta(\varrho)} \epsilon a^{-1}\sqrt{1-\varrho^2} \Big)\\				
				&\stackrel{\epsilon\to 0}{\sim}a^{p(\varrho)-1}\Big( \frac{1}{\sqrt{1-\varrho^2}} \Big)^{p(\varrho)}\frac{\pi}{\theta(\varrho)}\sqrt{1-\varrho^2}.
			\end{align*}		
			Plugging this calculation into (\ref{h1}), (\ref{h3}) the claim follows.
	\end{proof}
	In fact, it will always be sufficient to consider the case $a=1$ by simple scaling as we will see below in Lemma \ref{l}.
	\begin{definition}\label{jumpmeasure}
		The special case for $a=1$ will serve as basic jump measure. We abbreviate
		\begin{align*}
			\nu^\varrho(d(y_1,y_2))=\nu^\varrho_{(1,0)}(d(y_1,y_2))=\begin{cases}
				{p(\varrho)^2\sqrt{1-\varrho^2}}\frac{y_1^{p(\varrho)-1}}{\pi\big(y_1^{p(\varrho)}-1\big)^2}\,dy_1&: y_2=0,\\
				{p(\varrho)^2\sqrt{1-\varrho^2}} \frac{y_2^{p(\varrho)-1}}{\pi\big(y_2^{p(\varrho)}+1\big)^2}\,dy_2&: y_1=0,
				\end{cases}
		\end{align*}
	 	in the sequel.
	\end{definition}

	A quick glance at the explicit density of $\nu^\varrho$ shows that the densities are far from being symmetric for the $x$- and $y$-axis: a pole is found only at $(1,0)$. Moreover, the tail behavior shows that the measure restricted to the $y$-axis is finite and is infinite for the $x$-axis. This comes as no surprise from the definition: starting the correlated Brownian motions in $(1,\epsilon)$ and sending $\epsilon$ to zero forces the Brownian motions to exit the first 		quadrant closer and closer to the point $(1,0)$. The additional factor $1/\epsilon$ then leads to the pole at $(1,0)$.\\

	\smallskip	
	The reduction from $\nu^\varrho_{(a,0)}$ to $\nu^\varrho_{(1,0)}$ is motivated by the following scaling property.
	\begin{lem}\label{l}
		Suppose $f$ maps $E$ to $\R$ continuously, then
		\begin{align*}
			\int_E f(y_1,y_2)\,\nu_{(a,0)}^\varrho(d(y_1,y_2))&=\frac{1}{a}\int_E f(ay_1,ay_2)\,\nu^\varrho (d(y_1,y_2)).
		\end{align*}
	\end{lem}
	\begin{proof}
		Splitting $E$ in the two positive parts of the axes, the claim follows from a change of variables in the third line of the following computation:
		\begin{align*}
			&\quad\int_E f(y_1,y_2)\nu_{(a,0)}^\varrho(d(y_1,y_2))\\
			&=\int_E f(y_1,0)\nu_{(a,0)}^\varrho(d(y_1,0))+\int_E f(0,y_2)\nu_{(a,0)}^\varrho(d(0,y_2))\\
			&=\int_0^\infty f(y_1,0){p(\varrho)^2\sqrt{1-\varrho^2}} a^{p(\varrho)-1}\frac{y_1^{p(\varrho)-1}}{\pi\big(y_1^{p(\varrho)}-a^{p(\varrho)}\big)^2}\,dy_1\\
&\quad +\int_0^\infty f(0,y_2){p(\varrho)^2\sqrt{1-\varrho^2}} a^{p(\varrho)-1}\frac{y_2^{p(\varrho)-1}}{\pi\big(y_2^{p(\varrho)}+a^{p(\varrho)}\big)^2}\,dy_2\\
			&=\frac{1}{a}\int_0^\infty f(ay_1,0){p(\varrho)^2\sqrt{1-\varrho^2}}\frac{y_1^{p(\varrho)-1}}{\pi\big(y_1^{p(\varrho)}-1\big)^2}\,dy_1+\frac{1}{a}\int_0^\infty f(0,ay_2){p(\varrho)^2\sqrt{1-\varrho^2}} \frac{y_2^{p(\varrho)-1}}{\pi\big(y_2^{p(\varrho)}+1\big)^2}\,dy_2\\
			&=\frac{1}{a}\int_E f(ay_1,ay_2)\nu^\varrho (d(y_1,y_2)).
		\end{align*}
	\end{proof}

\subsubsection{Poissonian Integral Equations}\label{sec:spde}
	The aim of this section is to discuss properly the objects appearing in Theorem \ref{0} and to give elements of the proof. For convenience of the reader not familiar with jump diffusions we added a (very brief) summary to the appendix.\\
	Let $\mathcal N$ be a Poisson point process on $S\times E\times (0,\infty)\times (0,\infty)$ with intensity measure
	\begin{align}\label{N}
	\mathcal N'(\{k\},d(y_1,y_2),dr,ds)=  \nu^\varrho(d(y_1,y_2))\,dr\,ds, \;\;\forall k\in S.
\end{align}
The Poisson random measure ${\mathcal N}$ can be interpreted as a collection of independent  Poisson point measures 
$\{\mathcal N(\{k\},\cdot,\cdot,\cdot), k\in S\}$ on $E\times (0,\infty)\times (0,\infty)$  running independently at each site $k\in S$. 
	Then at each site $k\in S$, the basic jump measure $\nu^\varrho$ will be used to determine the target point of a jump from $E$ to $E$ at that site. To incorporate a state-dependent jump rate the $r$-component will be used. That is, as the jump intensity will depend on the current state of the system before time $s$, which we denote by $(U_{s-},V_{s-})$,  we define the intensities
	\begin{align}\label{intensity}
		I_s(k)&=\begin{cases}
		        	\frac{\mathcal AV_{s-}(k)}{U_{s-}(k)}&:U_{s-}(k)>0,\\
		        	\frac{\mathcal AU_{s-}(k)}{V_{s-}(k)}&:V_{s-}(k)>0,\\
			0&:U_{s-}(k)=V_{s-}(k)=0.
		       \end{cases}
	\end{align}
	Let us take a closer look at (\ref{intensity}), and assume for a moment that the current state at site $k$ is $(U_{s-}(k),0)$. Then the intensity of jumps at $k$ at time $s$ is high if $U_{s-}(k)$ is small compared to the total size of the population of ``type $V$" at neighboring sites.
	
	Next, we need to specify the integrand that describes the jumps of $\textrm{SBM}_\infty$ at an atom of ${\mathcal N}$ at $(k,(y_1,y_2),s,r)$:
	\begin{align*}
		J\big(y_1,y_2,U_{s-}(k),V_{s-}(k)\big)=y_2 \left(V_{s-}(k)\atop U_{s-}(k)\right)+(y_1-1)\left(U_{s-}(k)\atop V_{s-}(k)\right)
	\end{align*}
	so that at an atom $(k,(y_1,y_2),s,r)$ of $\mathcal N$ the system in state $(U_{s-},V_{s-})$ changes at site $k$ via one of the following transitions:
	\begin{align}\label{X}
		\begin{cases}
	                \left(U_{s-}(k)\atop 0\right)\mapsto \left(y_1U_{s-}(k)\atop 0\right)&:y=\left(y_1\atop 0\right),\\
&\\
                       	\left(U_{s-}(k)\atop 0\right)\mapsto \left(0\atop y_2 U_{s-}(k)\right)&:y=\left(0\atop y_2\right), \\
&\\
	                \left(0\atop V_{s-}(k)\right)\mapsto \left(0\atop y_1 V_{s-}(k)\right)&:y=\left(y_1\atop 0\right),\\
&\\
	                \left(0\atop V_{s-}(k)\right)\mapsto \left(y_2 V_{s-}(k)\atop 0\right)&:y=\left(0\atop y_2\right).
		\end{cases}
	\end{align}
	The second and fourth cases will be referred to as change of type as the jump changes the current state from one axis to the other. 
	Next, after compensating the jump term and adding an additional drift term, we are ready to define the system of Poissonian equations:	
	\begin{definition}		
		The system of Poissonian integral equations, indexed by the possibly infinite set $S$,
		\begin{align}\label{eqn:st}\begin{split}
			U_t(k)&=U_0(k)+\int_0^t\mathcal A U_s(k)\,ds\\
				&+\int_0^t\int_0^{I_s(k)}\int_{E}\big(y_2 V_{s-}(k)+(y_1-1)U_{s-}(k)\big)(\mathcal{N-N'})(\{k\},d(y_1,y_2),dr,ds),\\
			V_t(k)&=V_0(k)+\int_0^t\mathcal A V_s(k)\,ds\\
				&+\int_0^t\int_0^{I_s(k)}\int_{E}\big(y_2 U_{s-}(k)+(y_1-1)V_{s-}(k)\big)(\mathcal{N-N'})(\{k\},d(y_1,y_2),dr,ds),\end{split}
		\end{align}
		or in short
		\begin{align*}\begin{split}
		\left(U_t(k)\atop V_t(k)\right)&=\left(U_0(k)\atop V_0(k)\right)+\left(\int_0^t\mathcal{A}U_s(k)\,ds\atop \int_0^t\mathcal{A}V_s(k)\,ds\right)\\
		&\quad+\int_0^t\int_0^{I_s(k)}\int_E \left(y_2\left(V_{s-}(k)\atop V_{s-}(k)\right)+(y_1-1)\left(U_{s-}(k)\atop V_{s-}(k)\right)\right) (\mathcal{N-N'})(\{k\},d(y_1,y_2),dr,ds)		
		\end{split}		\end{align*}
		will be called \textbf{infinite rate symbiotic branching SPDE}.
	\end{definition}
	
	For a further discussion and connections of the above Poissonian SPDE to the standard voter process we refer to the next section.
	
	\begin{thm}\label{exis}
		Let $\varrho\in (-1,1)$ and suppose $(U_0,V_0)\in L^{2,E}_\beta$. Then Equation (\ref{eqn:st}) admits a weak solution with paths almost surely in $D([0,\infty),L_\beta^{2,E})$.
	\end{thm}
	\begin{proof}[Sketch of Proof]
	The proof is along the lines of Sections 2 and 3 of \cite{KM10b} for $\varrho=0$. The basic idea is to construct the approximating sequence of equations with the following modifications:
	\begin{enumerate}
		\item truncate the infinite index set (compare with the proof of Theorem \ref{ex:dis})),
		\item modify the jump measure $\nu^\varrho$ by truncating its jumps near the pole $(1,0)$. This makes the modified jump measure finite.
		\item modify the jump intensity $I$ by truncation its big values. 
	\end{enumerate}
	With these truncations there are only finitely many jumps up to any $t\geq 0$ so that solutions can be built merely ``by hands" via interlacing. To be more precise, we consider equations on subsets $S^m\subset S$ with $m$ 		elements and we redefine
	\begin{align*}
		\nu^{\epsilon,\varrho}(d(y_1,y_2))=\nu^\varrho(d(y_1,y_2)) \mathbf 1_{\{y_1-1>\epsilon, 1-y_1>\epsilon'\}}.
	\end{align*}
	This asymmetric truncation around $(1,0)$ is slightly strange but if $\epsilon,\epsilon'$ are chosen such that $\int_E (y_1-1)\nu^{\epsilon,\varrho}(d(y_1,y_2))=0$ then solutions stay on the boundary $E$ of the first quadrant, i.e. the drift does not push solutions into the interior. The modified state-dependent jump rate becomes  
	\begin{align*}
		I^\epsilon_t(k)&=\begin{cases}
		        	\frac{\mathcal AV^{m,\epsilon}_{t-}(k)}{U^{m,\epsilon}_{t-}(k)\vee \epsilon}&:U^{m,\epsilon}_{t-}(k)>0,\\
		        	\frac{\mathcal AU^{m,\epsilon}_{t-}(k)}{V^{m,\epsilon}_{t-}(k)\vee \epsilon}&:V^{m,\epsilon}_{t-}(k)>0,\\
			\frac{1}{\epsilon}\mathcal AU^{m,\epsilon}_{t-}(k)+\frac{1}{\epsilon}\mathcal AV^{m,\epsilon}_{t-}(k)&:U^{m,\epsilon}_{t-}(k)=V^{m,\epsilon}_{t-}(k)=0.
		       \end{cases}
	\end{align*}
	Replacing $\nu^\varrho$ by $\nu^{\epsilon,\varrho}$, $I$ by $I^\epsilon$ and adding additional jumps away from $(0,0)$, solutions $(U^{m,\epsilon},V^{m,\epsilon})$ can be constructed by hands via a Poisson point measure $\mathcal N$ since jumps do not accumulate. By definition it seems clear (ignoring the additional jumps away from zero which will vanish in the limit) that a possible limit for $m\to \infty$ and $\epsilon\to 0$ fulfills Equation (\ref{eqn:st}). This can be made rigorous via the 	method of characteristics for semimartingales and classical convergence theorems. To ensure that the sequence converges, tightness in the Skorohod space is justified by Aldous' criterion. Up to now the arguments copied directly those of \cite{KM10b}, with only difference in replacing $\nu^0$ by $\nu^\varrho$. To apply Aldous' criterion one then has to find $p$th moment estimates which again can be obtained as in \cite{KM10b} by replacing  their arguments relying on $p\in (1,2)=(1,p(0)
 )$ by the same arguments based on $p\in (1,p(\varrho))$.
\end{proof}	
	In order to clarify the connection of Equation (\ref{eqn:st}) to infinite rate symbiotic branching processes we give more elaborate arguments  for the next result which for $\varrho=0$ was proved in Lemma 3.12 of \cite{KM10b}.
	\begin{prop}\label{thm:s}
		Let $\varrho\in (-1,1)$ and suppose $(U,V)$ is a weak solution to Equation (\ref{eqn:st}) taking values in $L_\beta^{2,E}$, then $(U,V)$ is a solution of the martingale problem (\ref{mp}). 
	\end{prop}
	\begin{proof}[Sketch of Proof]
		To show that the martingale problem (\ref{mp}) is satisfied by weak solutions to (\ref{eqn:st}), one can proceed similarly to the proof of Proposition \ref{prop:mart1} by applying It\^o's formula to 								$F(U_t,V_t,z_1,z_2)$ for compactly supported $z_1,z_2$. Let us first define the integrands of the Poissonian integrals as
		\begin{align*}
			\left(J_1\big(y_1,y_2,U_{s-}(k),V_{s-}(k)\big)\atop J_2\big(y_1,y_2,U_{s-}(k),V_{s-}(k)\big)\right):=y_2 \left(V_{s-}(k)\atop U_{s-}(k)\right)+(y_1-1)\left(U_{s-}(k)\atop V_{s-}(k)\right)
		\end{align*}
		and abbreviate for $(x_1,x_2),(z_1,z_2)\in E^S$
		\begin{align*}
			&\quad \langle\langle x_1,x_2,z_1,z_2 \rangle\rangle_{\varrho,k}\\
			&=-\sqrt{1-\varrho}\big( x_1(k)+x_2(k)\big)\big(z_1(k)+z_2(k)\big)+i\sqrt{1+\varrho}\big( x_1(k)-x_2(k)\big)\big(z_1(k)-z_2(k)\big).
		\end{align*}
		 First, by It\^o's formula for non-continuous semimartingales and the notation for partial derivatives already used in Lemma \ref{l6}, we obtain for $(z_1,z_2)\in L^{f,E}$
		\begin{align*}
			&\quad F(U_t,V_t,z_1,z_2)\\
			&=F(U_0,V_0,z_1,z_2)+\sum_{k\in S}\int_0^t\frac{\partial}{\partial x_1(k)}e^{\langle\langle U_s,V_s,z_1,z_2\rangle\rangle_\varrho}\mathcal A U_s(k)\,ds\\
			&\quad+\sum_{k\in S}\int_0^t\frac{\partial}{\partial x_2(k)}e^{\langle\langle U_s,V_s,z_1,z_2\rangle\rangle_\varrho}\,\mathcal A V_s(k)\,ds\\
			&\quad+\sum_{k\in S}\int_0^t\int_0^{I_s(k)}\int_E\bigg[e^{\langle\langle (U_{s},V_{s})+J(y_1,y_2,U_{s},V_{s}),z_1,z_2\rangle\rangle_{\varrho,k}}-e^{\langle\langle U_{s},V_{s},z_1,z_2\rangle\rangle_{\varrho,k}}\\
			&\quad\quad-J_1(y_1,y_2,U_{s}(k),V_{s}(k))\frac{\partial}{\partial U(k)}e^{\langle\langle U_{s},V_{s},z_1,z_2\rangle\rangle_{\varrho,k}}\\
			&\quad\quad-J_2(y_1,y_2,U_{s}(k),V_{s}(k))\frac{\partial}{\partial V(k)}e^{\langle\langle U_{s},V_{s},z_1,z_2\rangle\rangle_{\varrho,k}}\bigg]\mathcal N'(\{k\},d(y_1,y_2),dr,ds)\\
			&\quad+\text{local martingale}
		\end{align*}
		which, carrying out the partial derivatives via Lemma \ref{l6} and plugging-in the definition of $\mathcal N'$, yields
		\begin{align*}
			&\quad F(U_t,V_t,z_1,z_2)\\
			&= F(U_0,V_0,z_1,z_2)+\int_0^t F(U_s,V_s,z_1,z_2)\langle\langle \mathcal A U_s, \mathcal A V_s,z_1,z_2\rangle\rangle_\varrho\,ds+\textrm{local martingale}\\
			&\quad+\sum_{k\in S}\int_0^te^{\langle \langle U_{s},V_{s},z_1,z_2\rangle\rangle_{\varrho,k}}I_s(k)
			\int_E \bigg[e^{\langle\langle J(y_1,y_2,U_{s},V_{s}),z_1,z_2\rangle\rangle_{\varrho,k}}-1\\
			&\quad\quad-\langle\langle J\big(y_1,y_2,U_{s},V_{s}\big),z_1,z_2\rangle\rangle_{\varrho,k}\bigg]\nu^\varrho(d(y_1,y_2))\,ds.
		\end{align*}
		The righthand side is already close to the martingale problem (\ref{mp}) if we can show that the sum of the integrals with respect to $\nu^\varrho(d(y_1,y_2))\,ds$ equals to zero and the local martingale is a martingale. Note that this is very similar to the proof of Proposition
		\ref{prop:mart1}. To prove the first assertion,  note that, by definition of $\nu^\varrho$,
		\begin{align*}
			&\quad\int_E \bigg[e^{\langle\langle J(y_1,y_2,U_{s},V_{s}),z_1,z_2\rangle\rangle_{\varrho,k}}-1-\langle\langle J(y_1,y_2,U_{s},V_{s}),z_1,z_2\rangle\rangle_{\varrho,k}\bigg]\nu^\varrho(d(y_1,y_2))\\
			&=\lim_{\epsilon\to 0} \frac{1}{\epsilon}E^{(1,\epsilon)}\left[e^{\langle\langle J(W^1_\tau,W^2_\tau,U_{s},V_{s}),z_1,z_2\rangle\rangle_{\varrho,k}}-1-\langle\langle J(W^1_\tau,W^2_\tau,U_{s},V_{s}),z_1,z_2\rangle\rangle_{\varrho,k}\right]
		\end{align*}
		so that we are done if we can show that, for any $(x_1,x_2),(z_1,z_2)\in E^S$ and $\epsilon>0$,
		\begin{align}\label{eh}
			E^{(1,\epsilon)}\left[e^{\langle\langle J(W^1_\tau,W^2_\tau,x_1,x_2),z_1,z_2\rangle\rangle_{\varrho,k}}-1-\langle\langle J(W^1_\tau,W^2_\tau,x_1,x_2),z_1,z_2\rangle\rangle_{\varrho,k}\right]=0.
		\end{align}
		But this identity holds, if $\tau$ is replaced by $t>0$, by It\^o's lemma as in the proof of Proposition \ref{prop:mart1}. The necessary arguments that justify the changes of limits and integration are as in Lemma 3.11 of \cite{KM10b}. Those incorporate the exit-time exit-point equivalence of Lemma \ref{ete} for $\varrho\neq 0$. The martingale property for the local martingale then follows as in the proof of Lemma 3.12 of \cite{KM10b} from first moment estimates that are not affected by $\varrho\in (-1,1)$.
	\end{proof}
	
	In the proof Proposition~\ref{thm:s} we did not utilize the particular form of the intensity $I_t(k)$ so that arbitrary changes in the intensity seem to lead to other solutions of the martingale problem, and by this seemingly imply a contradiction to uniqueness of the martingale problem (\ref{mp}). However, this chain of reasoning is not true because of the particular choice~(\ref{intensity}) for $I_t(k)$  forces solutions to have paths in $L_\beta^{2,E}$ and uniqueness for the martingale problem (\ref{mp}) only holds for solutions with paths restricted to $L_\beta^{2,E}$.\\	
	Let us make this more precise: suppose that $U_{T}(k)=0$ for some random time $T>0$ and some $k\in S$. From the density of the basic jump measure $\nu^\varrho$ it is clear that for some positive time no jump changing the types occurs
	(by finiteness of $\nu^\varrho$ restricted to the $y_2$, jumps that change types come with finite rate). Hence, for some positive random time $\delta$, no jump occurs so that
	\begin{align*}
		U_{T+r}(k)=0,\quad r\in [0,\delta].
	\end{align*}
	In particular, this shows that $\mathcal N'$ must be such that
	\begin{align*}
		\int_T^{T+r} \mathcal A U_s(k)\,ds-\int_T^{T+r}\int_0^{I_s(k)}\int_E J_1(y_1,y_2,U_{s-}(k),V_{s-}(k))\mathcal N'(\{k\},d(y_1,y_2),dr,ds)=0,
	\end{align*}
	for all $r\in [0,\delta].$	We now briefly show that the choice (\ref{intensity}) indeed does the job:
	\begin{align*}
		&\quad \int_T^{T+r}\int_E J_1(y_1,y_2,U_{s-}(k),V_{s-}(k))\,I_s(k)\,\nu^\varrho(d(y_1,y_2))\,ds\\
		&=\int_T^{T+r}\int_0^\infty y_2 V_{s-}(k)\frac{\mathcal AU_{s-}(k)}{V_{s-}(k)}{p(\varrho)^2}\sqrt{1-\varrho^2} \frac{y_2^{p(\varrho)-1}}{\pi\big(y_2^{p(\varrho)}+1\big)^2}\,dy_2\,ds\\
		&=\int_T^{T+r}{\mathcal AU_{s-}(k)}\,ds,
	\end{align*}
	because
	\begin{align*}
		\int_0^\infty y_2{p(\varrho)^2} \sqrt{1-\varrho^2}\frac{y_2^{p(\varrho)-1}}{\pi\big(y_2^{p(\varrho)}+1\big)^2}\,dy_2&=\lim_{\epsilon\to 0}\frac{1}{\epsilon}E^{1,\epsilon}\big[W^2_\tau\big]=\lim_{\epsilon\to 0}\frac{1}{\epsilon}\lim_{t\to\infty}E^{1,\epsilon}\big[W^2_{t\wedge \tau}\big]=\lim_{\epsilon\to 0}\frac{1}{\epsilon}\epsilon=1.
	\end{align*}
	Note that here the superscript in $W$ refers to the second coordinate of the pair of Brownian motions and not to the second moment.
	The first equality follows from the definition of $\nu^\varrho$; the second follows from the martingale convergence theorem for which the uniform integrability is ensured by the upper bound
	\begin{align*}
		E^{1,\epsilon}\big[\big(W_{t\wedge \tau}^2\big)^{p(\varrho)-\mu}\big]\leq E^{1,\epsilon}\big[\tau^{\frac{p(\varrho)-\mu}{2}}\big]<\infty,
	\end{align*}
	where the positive constant $\mu$ is chosen sufficiently small such that $p(\varrho)-\mu>1$ (existence of $\mu$ is ensured by the exit-time exit-point equivalence of Lemma \ref{ete}).\\
	
	With the Poissonian construction of $\textrm{SBM}_\infty$ in hand we now sketch a proof of Theorem \ref{pro:1}.
	
	\begin{proof}[Sketch of Proof for Theorem \ref{pro:1}]
		Existence of solutions to the martingale problem follows from Theorem~\ref{exis} and Proposition~\ref{thm:s}.\\
 The uniqueness proof is inspired by the proof of Lemma \ref{uniq} for $\gamma<\infty$ based on 		self-duality. Here, we sketch the chain of arguments of Section 4 in \cite{KM10b} which can be copied line by line while replacing the duality function in \cite{KM10b} by the $\varrho$-dependent duality function $F$ defined in (\ref{F}).\\
		\textbf{Step 1:} For compactly supported initial conditions $(\tilde U_0,\tilde V_0)$ solutions $(\tilde U,\tilde V)$ to the martingale problem are constructed via the Poissonian equations (\ref{eqn:st}). From the first moment estimates one obtains that solutions decay sufficiently fast at infinity. \\
		\textbf{Step 2:} First moment bounds for arbitrary solutions of the martingale problem are derived by differentiating the Laplace transform part (see Lemma 4.2 of \cite{KM10b} for $\varrho=0$). \\
		\textbf{Step 3:} The crucial part is to derive the self-duality relation
		\begin{align*}
			\E[F(U_t,V_t,\tilde U_0,\tilde V_0)]=\E\big[F(U_0,V_0,\tilde U_t,\tilde V_t)\big]
		\end{align*}
		between the two independent solutions $(U,V)$ and $(\tilde U,\tilde V)$ starting at $(U_0,V_0)\in L^{2,E}_\beta$ and $(\tilde U_0,\tilde V_0)\in L^{f,E}$. Now, as in the proof of Corollary \ref{uniq}, self-duality determines the one-dimensional laws $(U_t,V_t)$ along the lines of the proof of Proposition 4.7 in \cite{KM10b} for $\varrho=0$. Standard theory (see Theorem 4.4.2 of \cite{EK86}) allows us to extend the uniqueness of $1$-dimensional distributions to uniqueness of finite dimensional distributions. Finally, the strong Markov property for $(U,V)$ follows from measurability in the initial condition which is inherited from the finite jump rate approximation processes.
	\end{proof}
	
Combining Theorems~\ref{pro:1}, \ref{exis} and Proposition~\ref{thm:s} we immediately get the  following
theorem. 

\begin{thm}
\label{thm:111}
Let $\varrho\in(-1,1)$ and $(U_0,V_0)\in L^{2,E}_{\beta}$. Then there exists unique weak solution to~(\ref{eqn:st}) which  is  the unique solution to the 
martingale problem from Theorem~\ref{pro:1}.
\end{thm}


\subsection{Infinite Rate Symbiotic Branching Processes and Voter Processes II}\label{sec:voter2}
	The infinite rate symbiotic branching processes $\textrm{SBM}_\infty$ were characterized in previous subsections via various approaches. 
	In this final section we describe $\textrm{SBM}_\infty$ from the viewpoint of the standard voter process which is closely related to symbiotic branching with $\varrho=-1$ as we have already seen in the Section \ref{sec:voter1}.
	\smallskip
	
	For 	the rest of this section we stick to $\mathcal A=\Delta$ on $S=\Z^d$ for convenience.
	\smallskip
	
	We start with restating Theorem \ref{thm:2} for the $\varrho=-1$ case. However, note that we additionally have to assume $u^\gamma_0+v^\gamma_0\equiv 1$ since we cannot use the self-duality anymore as for $\varrho=-1$ it 	does not carry enough information to characterize the full law of the limiting process $(U,V)$. Under this additional assumption we can rely on the folklore results mentioned at the very end of Section \ref{sec:voter1} whereas for general initial conditions a different approach should be developed.
\begin{thm}\label{thm:12}
	Suppose $\varrho=-1$ and for any $\gamma>0$, $(u^{\gamma}_t,v^{\gamma}_t)_{t\geq 0}$ solves $\textrm{SBM}_{\gamma}(-1)$ and the initial condition $(u^\gamma_0,v^\gamma_0)=(U_0,V_0)$ do not depend on $\gamma$. If furthermore we suppose
	$$(U_0(k),V_0(k))\in \{(0,1),(0,1)\},\quad k\in\Z^d,$$
	then,	for any sequence $\gamma_n$ tending to infinity, we have the convergence in law
	\begin{align*}
		(u^{\gamma_n},v^{\gamma_n})\Longrightarrow (U,V),\quad n\to\infty,
	\end{align*}
	in $D([0,\infty),L_\beta^2)$ equipped with the Meyer-Zheng ``pseudo-path" topology. Here, $U$ is a standard voter process and $V=1-U$ its reciprocal voter process (i.e. opinions $1$ and $0$ are interchanged).
\end{thm}
\begin{convention}
\label{2907_1}
In what follows the pair of voter processes constructed in the above theorem will be called $\textrm{SBM}_\infty(-1)$.
\end{convention}
\begin{proof}[Sketch of Proof]
	As discussed in the end of Section \ref{sec:inter}, with the additional assumption on the initial conditions, $u^\gamma$ is a solution to the stepping stone model of Example \ref{ex1} and $v^\gamma=1-u^\gamma$. For $\gamma$ tending to infinity, a well-known result (see for instance Section 10.3.1 of \cite{dawson}) states that the finite dimensional distributions of solutions to the stepping stone model converge to those of the standard voter process; solutions are bounded and the moments converge as discussed in Section \ref{sec:voter1}. Tightness in the Meyer-Zheng ``pseudo-path" topology follows as for $\varrho\in (-1,1)$.
\end{proof}
	
	To understand $\textrm{SBM}_\infty$ and the voter process in a unified framework let us first summarize.
	{The infinite rate symbiotic branching processes $\textrm{SBM}_\infty(\varrho)$ are the weak limits of $\textrm{SBM}_\gamma(\varrho)$, as $\gamma\rightarrow \infty$,
		\begin{itemize}
			\item for $\varrho \in (-1,1)$, by Theorem \ref{thm:2},
			\item for $\varrho=-1$ and $U_0+V_0=\mathbf{1}$, by Theorem \ref{thm:12}.
		\end{itemize}
	} 
	\smallskip
		
	 A unified representation can be given with the Poissonian approach developed above if $\nu^\varrho$ is extended to $\varrho=-1$ as
	 \begin{align*}
	 	\nu^{-1}(d(v_1,v_2))=\delta_{(0,1)}(v_1,v_2).
	 \end{align*}
	With the intensities $I_s(k)$ defined in (\ref{intensity}) and the Poisson point processes $\mathcal N$ with intensity measure $\mathcal N'$ as in (\ref{N}) we can extend Theorem~\ref{thm:111} as follows:
	
	\begin{thm}
		Suppose $\varrho\in [-1,1)$, $(U_0,V_0)\in L^{2,E}_\beta $ and for $\varrho=-1$ assume additionally $U_0+V_0\equiv 1$. Then the infinite rate symbiotic branching process  $(U,V)$ with initial condition $(U_0,V_0)$ 			coincides in law with the unique weak solution to (\ref{eqn:st}).
	\end{thm}
	Note that the additional assumption on the initial condition is not necessary for  Equation (\ref{eqn:st}) to have weak solutions. We believe that also the convergence of $\textrm{SBM}_\gamma(-1)$ to the solutions of (\ref{eqn:st}) holds without the restriction.
	\begin{proof}
		For the case $\varrho\in (-1,1)$ the theorem is nothing else but Theorem \ref{thm:111} so that we only need to discuss the extension to $\varrho=-1$.\\
	Existence of a weak solution to~(\ref{eqn:st}), for $\varrho=-1$, can be verified as sketched in the proof of Theorem \ref{exis} for $\varrho\in (-1,1)$; since the jump measure $\nu^{-1}$ is finite the proof is simpler since no truncation procedure for $\nu^{-1}$ is needed.\\
	To identify the weak solutions  to~(\ref{eqn:st}) with $\textrm{SBM}_\infty(-1)$ it suffices, by Theorem \ref{thm:12}, to show that, for any weak solution $(U,V)$ to~(\ref{eqn:st}),  $U$ is a voter process and $V=1-U$.
	 We use two facts: first, the jumps preserve the property $(U_t(k), V_t(k))\in \{(1,0),(0,1)\}$ for all $t\geq 0, k\in \Z^d$ and, secondly, the drift and the compensator integral cancel each other. To establish the first, note that 
	the choice of $\nu^{-1}$ implies that always $y_1=0$ and $y_2=1$ so that the 		only transitions are (compare with (\ref{X}))
	\begin{align*}
	                \left(U_{s-}(k)\atop 0\right)&\mapsto \left(0\atop U_{s-}(k)\right),\\
                       	\left(0\atop V_{s-}(k)\right)&\mapsto \left(V_{s-}(k)\atop 0\right),
		\end{align*}
	or, simply,
	\begin{align*}
	                \left(1 \atop 0\right)&\mapsto \left(0\atop 1\right),\\
                       	\left(0\atop 1\right)&\mapsto \left(1\atop 0\right).
		\end{align*}
		
	The latter follows from the simple computation
			\begin{align*}
			&\quad\int_0^t\int_0^{I_s(k)} \int_E y_2\left(V_{s-}(k)\atop U_{s-}(k)\right)+(y_1-1)\left(U_{s-}(k)\atop V_{s-}(k)\right) \mathcal{N'}(\{k\},d(y_1,y_2),dr,ds)\\
				&=\int_0^t \left(V_{s}(k)-U_{s}(k)\atop U_{s}(k)-V_{s}(k)\right) \left(\frac{\Delta U_{s}(k) }{V_{s}(k)}\mathbf 1_{\{U_{s}(k)=0\}}+\frac{\Delta V_{s}(k)}{U_{s}(k)}\mathbf 1_{\{V_{s}(k)= 0\}}\right)ds\\
				&=\left(\int_0^t{\Delta U_{s}(k) }\mathbf 1_{\{U_s(k)=0\}}-{\Delta V_{s}(k) }\mathbf 1_{\{V_s(k)=0\}}\,ds\atop \int_0^t{-\Delta U_{s}(k) }\mathbf 1_{\{U_s(k)=0\}}+{\Delta V_{s}(k) }\mathbf 1_{\{V_s(k)=0\}}\,ds\right)\\
			&=\left(\int_0^t{\Delta U_{s}(k) }\,ds\atop \int_0^t{\Delta V_{s}(k)}\,ds\right)
		\end{align*}
		for which we used $U_s(k), V_s(k)\in \{0,1\}$ and
		\begin{align*}
			\Delta U_s(k)+\Delta V_s(k)=\sum_{|j-k|=1}\frac{1}{2d} \big(U_s(j)+V_s(j)\big)-\big(U_s(k)+V_s(k)\big)=\sum_{|j-k|=1}\frac{1}{2d}-1=0.
		\end{align*}
		Hence, canceling the compensator integral with the drift shows that Equation (\ref{eqn:st}) can be written equivalently in the simplified form
	\begin{align*}
		\left(U_t(k)\atop V_t(k)\right)=\left(U_0(k)\atop V_0(k)\right)+\int_0^t\int_0^{I_s(k)}\int_E\left( \left(V_{s-}\atop U_{s-}\right)-\left(U_{s-}\atop V_{s-}\right) \right)\mathcal{N}(\{k\},d(y_1,y_2),dr,ds).
	\end{align*}
	Since the configurations only change by a jump and the jumps only switch $0$ to $1$ and vice versa one can already guess that both coordinates are reciprocal voter processes. To make this precise we apply It\^o's formula to functions of $(U_t,V_t)$ and derive that $(U,V)$ satisfies the martingale problem for the standard voter process. It suffices to carry this out for $U$
since we already know that $V_t=1-U_t\,,$ for all $t\geq 0$. \\
 Let us fix a test-function $f:\{(0,1)\}^{\Z^d}\to \R$ that only depends on finitely many coordinates $k\in K$, $\# K<\infty$, and apply It\^o's formula to $f(U_t)$
 to obtain
	\begin{align*}
		f( U_t)&=f(U_0)+\sum_{k\in K}\int_0^t\int_0^{I_s(k)}\int_E\left[f\big(( U_{s-})^{(k)}\big)-f\big( U_{s-}\big)\right]\mathcal N(\{k\},d(y_1,y_2),dr,ds).
	\end{align*}
	We denoted again by $\eta^{(k)}$ the configuration that is obtained from the configuration $\eta$ flipping only the opinion at site $k$. Adding and subtracting the compensated integral leads to
	\begin{align*}
		f( U_t)
		&=f(U_0)+\sum_{k\in K}\int_0^t\int_0^{I_s(k)}\int_E \left[f\big(( U_{s-})^{(k)}\big)-f\big( U_{s-}\big)\right]\mathcal{(N-N')}(\{k\},d(y_1,y_2),dr,ds)\\
		&\quad+\sum_{k\in K}\int_0^tI_s(k)\left[f\big(( U_{s-})^{(k)}\big)-f\big( U_{s-}\big)\right]\,ds.
	\end{align*}
	 Next, we use that for all $s\geq 0, k\in \Z^d$ we have $U_{s-}(k), V_{s-}(k)\in \{0,1\}$ to obtain
	\begin{align*}
		I_s(k)&=\begin{cases}
		        	\frac{\Delta V_{s-}(k)}{U_{s-}(k)}&:U_{s-}(k)>0\\
		        	\frac{\Delta U_{s-}(k)}{V_{s-}(k)}&:V_{s-}(k)>0
		       \end{cases}\\
		       &=\begin{cases}
		        	\Delta V_{s-}(k)&:U_{s-}(k)=1\\
		        	\Delta U_{s-}(k)&:U_{s-}(k)=0
		       \end{cases}\\
		       &=\frac{1}{2d}\#\big\{\text{neighbors of the voter at }k \text{ who have an opinion different than his at time  } {s-}\big\}\\
		       &=c(k, U_s).
	\end{align*}
	Plugging-in, we proved that
	\begin{align*}
		M_t^f:=f(U_t)-f(U_0)-\int_0^t\sum_{k\in K}c(k,U_s)\left[f\big(( U_{s})^{(k)}\big)-f\big( U_{s}\big)\right]ds
	\end{align*}
	is a local martingale and since everything is bounded it is, in fact,  a martingale. This shows that $U_t$ has the generator (\ref{gener}) of the voter process.\\
	Well-posedness for this martingale problem implies the weak uniqueness statement of the theorem for $\varrho=-1$.
\end{proof}
	
	Finally, we want to explain that the extended choice of $\nu^\varrho$ is more natural than it appears on first view. There are two good reasons. First, going back to Definitions \ref{defj} and \ref{jumpmeasure} let us see what we get for $\varrho=-1$:
		\begin{align*}
			\lim_{\epsilon\to 0} \frac{Q^{-1}_{(1,\epsilon)}}{\epsilon}=\lim_{\epsilon\to 0}\frac{1}{\epsilon}\left(\frac{\epsilon}{1+\epsilon}\delta_{(0,1+\epsilon)}+\frac{1}{1+\epsilon}\delta_{(1+\epsilon,0)}\right)=\delta_{(0,1)}+\infty \delta_{(1,0)},
	\end{align*}
	since for completely negatively correlated Brownian motions $(B^1,B^2)$ started at $(u,v)$ the exit-measure from the first quadrant is $\frac{v}{u+v}\delta_{(0,u+v)}+\frac{u}{u+v} \delta_{(u+v,0)}$. Secondly, a more careful look at the density of $\nu^\varrho$ for $\varrho\in (-1,1)$ shows that the mass accumulates at $(1,0)$ and $(0,1)$ since $p(\varrho)$ explodes for $\varrho$ tending to $-1$. More precisely, $\nu^\varrho$ converges in the vague topology (extended to the completion of $\R$) to $\delta_{(0,1)}+\infty\delta_{(1,0)}$.	
	Unfortunately, both justifications lead to $\nu^{-1}$ with an additional infinite atom at $(1,0)$. 
	Luckily, the infinite atom at $(1,0)$ has no impact on the Poissonian equations since the integrand of (\ref{eqn:st}) vanishes if $y_2=0$ 	and $y_1=1$. We believe that some rigorous work on this observation might lead to some interesting results.\\
	
\textbf{This brief discussion explains the natural unification of the family $\textrm{SBM}_\infty$ with the voter process at its boundary $\varrho=-1$ and justifies our interpretation of $\textrm{SBM}_\infty(\varrho)$ as generalized voter process, given below Theorem \ref{0}.}
	
\section*{Acknowledgement}	
	LD acknowledges an ESF Grant ``Random Geometry of Large Interacting Systems and Statistical Physics" and hospitality of the Technion. LM acknowledges hospitality of the Universit\'e Paris 6.


\appendix

\section{A Very Rough Primer on Jump SDEs}
	Symbiotic branching models are by definition solutions of (possibly infinite) systems of ordinary stochastic differential equations
	\begin{align}\label{dif}
		dX_t=b(X_t)dt+\sigma(X_t)dB_t.
	\end{align}
	Interestingly, the infinite rate analogues that have been defined so far as solutions to exponential martingale problems can be represented as solutions to jump-type stochastic differential equations. The most straight-forward 		generalization of (\ref{dif}) is
	\begin{align}\label{dif2}
		dX_t=b(X_t)dt+\sigma(X_t)dB_t+c(X_{t-})dL_t
	\end{align}
	for a L\'evy process $L_t$. The modeling drawback of (\ref{dif2}) is that once $L_t$ has a jump $x$, then $X_t$ has a jump $c(X_{t-})x$. If the jumps of the solution process are meant to depend on the jumps of the jump-measure 	in a non-linear way, other concepts are needed. One way to model such processes is to replace the jump noise by a general compensated random measure:
	\begin{align*}
	 	dX_t=b(X_t)dt+\sigma(X_t)dB_t+c(X_{t-},x)(\mathcal N-\mathcal N')(dt,dx).
	\end{align*}
	This notion of jump-type stochastic differential equation is needed for our purposes. Unfortunately, the basic jump measure $\nu^\varrho$ of Theorem \ref{0} has a second order singularity at $(1,0)$ and a polynomial 		decreasing tail which for $\varrho\geq 0$ prevents existence of second moments. This causes the general second moment integration theory to collapse here and the abstract martingale integration theory with respect to compensated random measures comes into play. To guide the reader unfamiliar with those concepts we briefly recall some core definitions and concepts.
	
	\medskip
	First, suppose $\mathcal N(dt,dx)(\omega)$ is a Poisson point measure on $[0,\infty)\times \R^d$ with compensator measure $\lambda$ on a stochastic basis $(\Omega, \mathcal F, (\mathcal F_t)_{t\geq 0}, \P)$, i.e. for all measurable sets $A$ with $\lambda(A)<\infty$, $\mathcal N([0,t],A)$ is a Poisson process in $t$ with parameter $\lambda(A)$ such that for disjoint sets $A_1,A_2$ the processes $\mathcal N([0,t],A_i)$ are independent. Defining $\mathcal N'([0,t],A)=t\lambda(A)$, it then follows that the compensated process
	\begin{align}\label{m}
		(\mathcal{N-N'})([0,t],A):=\mathcal N([0,t],A)-\mathcal N'([0,t],A)
	\end{align}
	is a martingale.	This property motivates the name martingale measure for the random measure $\mathcal N-\mathcal N'$. Given a predictable integrand $H(s,x)(\omega)$, defined on the stochastic basis of the driving point 	process, one then aims to define the integral process
	\begin{align*}
		\int_0^t\int_{\R^d}H(s,x)(\mathcal N-\mathcal N')(ds,dx)
	\end{align*}
	of $H$ against the compensated martingale measure $\mathcal N-\mathcal N'$ via an $L^2$-approximation procedure for integrands in the space
	\begin{align*}
		\mathcal H=\left\{H:\E\left[\int_0^t\int_{\R^d}H(s,x)^2\mathcal N'(ds,dx)\right]<\infty\right\}. 
	\end{align*}
	Integrals are then defined as limits of integrals of simple processes against $\mathcal N-N'$. The martingale property of the driving measure implies that the stochastic integral itself is a square-integrable martingale.
	For a more detailed introduction we refer the reader for instance to the overview article \cite{B} or \cite{IW}.
	
	\medskip	
	Unfortunately, it turns out that for $\varrho\geq 0$ our basic jump measure $\nu^\varrho$ has too heavy tails so that the integrands fail to be members of $\mathcal H$. Fortunately, abstract martingale theory allows for an integration theory with respect to compensated random measures without requiring $H\in \mathcal H$. The rest of this section consists of a short summary of the integration theory developed in Section II.1d of \cite{JS07}. To make our life simpler (and this is what we need) we assume that the appearing compensator measure is absolutely continuous in $t$ so that the presentation is slightly simplified in contrast to the general theory presented in \cite{JS07}.\\	
	Suppose that for a subspace $E$ of $\R^n$, $\mathcal N$ is an integer valued random measure (not necessarily Poissonian) on $[0,\infty)\times E$, i.e. a family of measures $\mathcal N(dt,dx)(\omega)$ on $([0,\infty)\times E, \mathcal B\otimes \mathcal E)$ such that $\mathcal N(\{0\}\times E)(\omega)=0$ almost surely, i.e. no jump at time $0$, and that $\mathcal N(\cdot)$ is an integer. 
	Building upon (\ref{m}) the concept of a compensator measure for general random point processes is generalized as follows: $\mathcal N'$ is the up to a null set unique (now possibly random) measure such that
	\begin{align}\label{abcd}
		W\ast \mathcal N_t-W\ast \mathcal N'_t
	\end{align}
	is a martingale null at zero for a suitably class of test-functions $W$. Here, $\cdot \ast\cdot_t$ stands for pathwise Lebesgue integration on $E\times [0,t]$. As, by assumption, the jump measure $\mathcal N$ is integer valued it should come as no surprise that $\mathcal N$ may be regarded as counting measure for the jumps of an auxiliary $E$-valued optional process $\beta_t$, i.e.
	\begin{align*}
		\mathcal N([0,t]\times A)(\omega)=\sum_{s\leq t} \mathbf 1_{A}(\Delta\beta_s(\omega)).
	\end{align*}	
	 With this notation in hand we can proceed with the abstract definition of the stochastic integral (see Definition II.1.27b) of \cite{JS07}). Absolute continuity in time of the compensator implies that $\mathcal N'(\{t\}\times dx)(\omega)=0$ almost surely so that the quantity $\hat W$ in \cite{JS07} vanishes. The set of possible integrands is changed to 
	\begin{align*}
		\mathcal  G&=\Bigg\{H:\E\Big[\sum_{s\leq t}H^2(s,\Delta \beta_s)\mathbf 1_{\{\Delta \beta_s\neq 0\}}\Big]^{1/2}<\infty\Bigg\}
	\end{align*}
	and the stochastic integral 
	\begin{align*}
		H\ast (\mathcal N-\mathcal N')_t=\int_0^t\int_E H(s,x)(\mathcal N-\mathcal N')(ds,dx) 
	\end{align*}
	is defined to be the unique (up to indistinguishable) purely discontinuous local martingale $X_t$ such that 
	\begin{align}\label{rem:1}
		\Delta X_\cdot\quad \text{and}\quad H(\cdot,\Delta \beta_\cdot)\mathbf 1_{\{\Delta \beta_\cdot\neq 0\}}\quad\text{are indistinguishable.}
	\end{align}
	Hence, if $\mathcal N$ has an atom at $(s,x)$, the stochastic integral $H\ast (\mathcal N-\mathcal N')$ has a jump $H(s-,x)$.
	Recall that by definition a purely discontinuous local martingale is required to be orthogonal to all continuous martingales but not to be pathwise everywhere discontinuous. For example, if $N_t$ is a standard Poisson process, the 	compensated process $N_t-t$ is purely discontinuous but far from being pathwise everywhere discontinuous.\\
	The integrability condition for class $\mathcal G$ is rather unsatisfactory as it involves the jump measure itself rather than only its compensator which might be more easy to handle. A characterization of the set $\mathcal G$ is 	given in Theorem II.1.33 of \cite{JS07}: it suffices to show that (recall that in our setting $\hat W$ of \cite{JS07} vanishes)
	\begin{align}\label{nec}\begin{split}
		\E\left[ \int_0^t \int _E H^2(s,x)\1_{\{|H(s,x)|\leq 1\}} \mathcal N'(ds,dx)\right]<\infty,\\
		\E\left[ \int_0^t \int _E |H(s,x)|\1_{\{|H(s,x)|\geq 1\}} \mathcal N'(ds,dx)\right]<\infty,\end{split}
	\end{align}
	showing in particular that $\mathcal G\subset \mathcal H$.	Finally, to motivate the naming ``stochastic integral'' for the abstract local martingale $H\ast (\mathcal N-\mathcal N')_t$, the following property should be mentioned. If the integrand is nice, that is,  additionally $\E[|H| \ast \mathcal N'_t]<\infty$, then both 	integrals against $\mathcal N$ and the compensator measure $\mathcal N'$ can be defined pathwise and
	\begin{align}\label{he}
		H\ast (\mathcal N-\mathcal N')_t=H\ast \mathcal N_t-H\ast \mathcal N'_t.
	\end{align}

\end{document}